\documentclass[a4paper,11pt]{article}

\usepackage[english]{babel}
\usepackage[T1]{fontenc}
\usepackage[utf8]{inputenc}
\usepackage{amsmath,amssymb,mathrsfs, tensor}
\usepackage{amsthm}
\usepackage{amscd}

\usepackage{comment}
\usepackage{tikz}
\usepackage{pgfplots}
\pgfplotsset{compat=1.18}
\usepackage{subcaption}

\usepackage[colorlinks=true,linkcolor= blue, citecolor=blue]{hyperref}
\usepackage{a4wide}
\numberwithin{equation}{section}

\newtheorem{theo}{Theorem}[section]
\newtheorem{ass}[theo]{Assumption}

\newtheorem{prop}[theo]{Proposition}
\newtheorem{lem}[theo]{Lemma} 
\newtheorem{exam}[theo]{Example}
\newtheorem{defi}[theo]{Definition}
\newtheorem{rem}[theo]{Remark} 
\numberwithin{figure}{section}

\usepackage{parskip} 
\makeatletter
\def\thm@space@setup{%
  \thm@preskip=\parskip \thm@postskip=0pt
}
\makeatother
\usepackage{float}

\date{}

\let\originalleft\left
\let\originalright\right
\renewcommand{\left}{\mathopen{}\mathclose\bgroup\originalleft}
\renewcommand{\right}{\aftergroup\egroup\originalright}

\usepackage{bbm}

\def\B{\mathcal{B}}

\def\C{\mathbb{C}}
\def\Cc{\mathcal{C}}
\def\D{\mathcal{D}}
\def\E{\mathcal{E}}
\def\eps{\varepsilon}
\def\F{\mathcal{F}}

\def\G{\mathbb{G}}
\def\H{\mathcal{H}}
\def\Hh{\mathscr{H}}

\def\K{\mathcal{K}}

\def\M{\mathcal{M}}
\def\N{\mathbb{N}}
\def\P{\mathcal{P}}

\def\Qq{\mathcal{Q}}
\def\R{\mathbb{R}}
\def\S{\mathcal{S}}
\def\T{\mathbb{T}}
\def\Tau{\mathcal{T}}
\def\Z{\mathbb{Z}}

\DeclareMathOperator\supp{supp}

\DeclareMathOperator{\ess}{ess}
\DeclareMathOperator{\ac}{\sigma_{ac}}

\DeclareMathOperator{\sic}{\sigma_{sc}}

\DeclareMathOperator{\Rp}{Re}
\DeclareMathOperator{\Ip}{Im}
\newcommand\con[1]{\overline{#1}}
\newcommand\innp[2]{\left\langle {#1}, {#2} \right\rangle}

\newcommand{\bra}{\langle}
\newcommand{\ket}{\rangle }

\newcommand\at[2]{\left.#1\right|_{#2}}

\usepackage{mathtools}

\DeclarePairedDelimiter\floor{\lfloor}{\rfloor}

\title{On the spectral properties of long-range perturbations of a class of block finite difference operators}
\author{Olivier Bourget\footnote{Partially supported by Fondecyt grant 1211576}  \and Angela Vargas-Mancipe \footnote{Supported by ANID Doctorado Nacional 21201703}}

\begin{document}

\maketitle

\vspace*{-0.5cm}
\noindent
Facultad de Matem\'aticas, Pontificia Universidad Cat\'olica
 de Chile, Vicu\~na Mackenna 4860, Santiago, Chile \\

\begin{abstract}

We analyze spectral properties of a family of self-adjoint first-order finite difference operators acting on $\ell^2(\Z; \C^2)$ or  $\ell^2(\Z_+; \C^2)$.
Applying the conjugate operator method, we prove the existence of limiting absorption principles and the absence of singular continuous spectrum for these operators. Our results cover classes of admissible long-range perturbations that have not been previously addressed. As illustrative examples, one-dimensional discrete Dirac operators and the Su–Schrieffer–Heeger (SSH) model are considered.

\end{abstract}

\tableofcontents

\section{Introduction}\label{intro}

The spectral analysis of block finite difference operators has attracted significant interest due to its relevance in both theoretical and applied contexts. For example, a massless two-dimensional discrete Dirac operator on a hexagonal lattice can be used to model electron transport in graphene at low excitation energies \cite{cn-g-per,n-7}. A simpler model is obtained by considering a discrete Schrödinger operator on the same lattice, see e.g., \cite{T}. The spectral properties of discrete Dirac operators have been extensively studied, see \cite{ales, b-s, b-mf-t, c-d-p2, c-i-k-s,  c-g-j,  hul, k-t,Nakamura2024} and references therein. In particular, \cite{Nakamura2024} discusses convergence of discrete Dirac operators to their continuous counterparts in $\R^n$, as the mesh size tends to zero.

Block Toeplitz and Jacobi operators provide another framework for the analysis of models arising from physics,  such as the Su-Schrieffer-Heeger (SSH) model, which is commonly used to describe polyacetylene chains \cite{SSH79}. Details on the spectral properties of these block operators can be found in \cite{B-Sah2, P.C, P.C2,  P.C-sah, sahahah} and references therein.

In scattering regimes,  the above operators do not exhibit singular continuous spectrum, while their absolutely continuous and pure point components can be effectively controlled. The natural question that arises is whether these spectral properties are stable under compact self-adjoint perturbations. A standard approach to this problem is to control the point spectrum and obtain a Limiting Absorption Principle (LAP) on appropriate subsets of the spectrum of the operator under consideration. This can be accomplished by positive commutator methods, such as the Mourre theory  \cite{abmg,EM}. This theory was applied to block Jacobi operators in  \cite{sahahah}. Later, \cite{P.C-sah} implemented the same method for block Toeplitz operators, generalizing previous results for the scalar case \cite{o-toe, B-Sah2}. For one-dimensional discrete Dirac operators, a LAP was obtained in \cite{g-h, kopylova, k-t2,  k-t} by different techniques.

In this paper, we apply the Mourre theory to examine the stability of spectral properties of certain first-order block finite difference operators under suitable compact self-adjoint perturbations. Our analysis focuses on the treatment of long-range perturbations, extending previous results for these operators. Although some of our results naturally extend to block finite difference operators on $\ell^2(\Z^d; \C^N)$ or $\ell^2(\Z^d_+; \C^N)$, we restrict our attention to the case $d=1$ and $N=2$. The corresponding extensions will be addressed in a forthcoming work. As applications, we consider the SSH model and the one-dimensional Dirac operators for both the massive and the massless cases. 

This article is organized as follows. In Section \ref{models}, we introduce the basic model and present its fundamental spectral properties. In Section \ref{Perturbations}, we introduce certain subspaces of matrix-valued sequences, which are needed to formulate our main results in Theorems   \ref{main1} and \ref{main2}. In Section \ref{connection}, we discuss the connection between our results and previous works. In particular, Proposition \ref{counterex} states that there exist long-range perturbations of the one-dimensional discrete Dirac operator that satisfy the hypotheses of our main theorems, but do not meet the conditions introduced in \cite{ g-h}. Section \ref{ReviewMourre} reviews the main features of Mourre theory. In Section~\ref{MourreEstimatesFreeModel} we construct a sequence of conjugate operators for the unperturbed operator, in the sense of the Mourre theory, based on ideas presented in \cite{g-m-2}. Furthermore, in Section~\ref{A0Gapless} we define an alternative conjugate operator in the absence of spectral gaps, which allows us to handle conical intersection of the spectral band functions. In Section \ref{admissiblesboth}, we study the classes of admissible perturbations with respect to these conjugate operators. The proofs of Theorems~\ref{main1} and \ref{main2} are developed in Section \ref{proofss}.

{\bf Notation.} Throughout this paper, ${\mathbb Z}$, ${\mathbb Z}_+$ and  ${\mathbb Z}_-$ denote the sets of integers, non-negative integers, and negative integers, respectively. If $x \in \R$, then $\floor{x}$ denotes the greatest integer less than or equal to $x$.  We also let $ \T$ be the one-dimensional torus $\T := \R /2 \pi\Z$, and $M_d(\C)$ states for the set of all $d \times d$ complex matrices. If   $T \in M_d(\C)$, then $T^*$ is its conjugate transpose.  

Given a separable complex Hilbert space ${\Hh}$, we denote by $\mathcal{B}(\Hh)$ the $C^*$-algebra of bounded linear operators acting on $\Hh$ and by $\K(\Hh)$ the ideal of compact operators. For any self-adjoint operator $H \in \mathcal{B}(\Hh)$, we write ${\mathcal E}_{\textit{p}} (H)$ for the set of its eigenvalues.  We set $E_I(H)$ for the spectral measure of $H$ on the Borel set $I \subseteq \R$. The spectrum of $H$ is denoted by $\sigma(H)$, and its essential spectrum by $\sigma_{\ess}(H)$. The sets $\ac(H)$ and $\sic(H)$ refer to the absolutely continuous and singular continuous components of $\sigma(H)$, respectively.

If $A$ is a self-adjoint operator on ${\Hh}$ with domain $\D(A)$,  we define $\bra A \ket:= \sqrt{A^2+1}$, and for $s \in \R$  we denote by $\Hh_s(A)$ the Sobolev space associated to $A$. We will identify $\Hh$ with its antidual $\Hh^*$ via the Riesz isomorphism. So, for $s \geq 0$,  $\Hh_s(A)$ is $\D( \bra A \ket ^s)$ equipped with the norm $\| \psi \|_s := \| \bra A \ket^s \psi\|$, and $\Hh_{-s}(A)$ is the completion of $\Hh$ with respect to the norm $\| \psi \|_{-s} := \| \bra A \ket^{-s} \psi\|$. If $0 \leq s \leq t$, then $\Hh_{-s}(A)$ can be identified with  $( \Hh_s(A))^{*}$ and we have the continuous embeddings
$$ \Hh_{ t}(A) \subseteq \Hh_{ s}(A) \subseteq  \Hh \subseteq \Hh_{-s}(A) \subseteq \Hh_{-t}(A),$$
and
$$ \B(\Hh) \subseteq \B \left(\Hh_s(A), \Hh_{-s}(A)\right)  \subseteq  \B \left(\Hh_t(A), \Hh_{-t}(A) \right).$$ 

We refer to \cite{abmg} for more properties of these spaces.


\section{General framework and main results}\label{models}

The aim of this section is to introduce the main results of this paper, namely Theorems \ref{main1} and \ref{main2}, and to discuss the relation with previous results in the literature (see Section \ref{connection}). We start by describing the unperturbed operator.

\subsection{The unperturbed model} \label{FreeModel}

In this section, we introduce the unperturbed operator, which plays a central role in the subsequent analysis. Let $\G \in \{\Z, \Z_{+}\}$.  We consider the Hilbert space $\H_\G := \ell^2(\G; \C^2)$, equipped with the standard inner product
\begin{align*}
\innp{u}{v}_\G& :=  \sum_{n\in \G} \bra u(n), v(n) \ket_{{\mathbb C}^2}, \quad u,v \in \H_\G.
\end{align*}
Furthermore, we define $\ell^2_{c}(\G; \C^2)$ as the subspace of $ \H_\G$ consisting of sequences with compact support.

Let $\alpha \in \R$ and $ a,b \in \C \setminus\{0\}$. The first-order finite difference operator $H_0^{(\G)}=H_0^{(\G)}(\alpha,a,b)$ on $\H_{\G}$,  determined by $\alpha, a$ and $b$ is defined as
\begin{equation}\label{DefH0}
    H_0^{(\G)}:=  \left(
\begin{array}{cc}
\alpha & \bar{a}+\bar{b}S^* \\
a+bS & -\alpha
\end{array} \right),
\end{equation}
where $S \in \B\left(\ell^2(\G)\right)$ is the right-shift operator on $\ell^2(\G) $, that is,
$$(Su)(n) := \left\{ 
\begin{array}{ll}
    0 & \mbox{if }  \G= \Z_+ \mbox{ and } n=0, \\
     u(n-1) & \mbox{otherwise},    \end{array}
\right. $$
for $u \in \ell^2(\G)$ and $n \in \G$. Clearly, $H_0^{(\G)}$ is a bounded self-adjoint operator on $\H_{\G}$.

\begin{rem}\label{rmkH1}
\begin{itemize}
    \item[(i)] The spaces $\H_{\Z_{+}}$ and $\H_{\Z_{-}}:= \ell^2(\Z_{-}; \C^2)$ are identified as subspaces of $\H_{\Z}$ such that $\H_{\Z} = \H_{\Z_{-}} \oplus \H_{\Z_{+}} $. We let $P: \H_{\Z} \to \H_{\Z_{+}}$ be the orthogonal projection of $ \H_{\Z}$ onto $ \H_{\Z_{+}}$. Consequently, $H_0^{(\Z_+)}$ can be identified with $\at{P H_0^{(\Z)} P}{\H_{\Z_{+}}} :\H_{\Z_{+}} \to \H_{\Z_{+}}  $. 
    \item[(ii)]  We could consider the first-order finite difference operator on $\H_{\G}$ given by
    $$  \left(
\begin{array}{cc}
c & \bar{a}+\bar{b}S^* \\
a+bS & d 
\end{array} \right),$$
where $a,b \in \C \setminus \{0\}$ and $c,d \in \R$.
This operator agrees with $\mu + H_0^{(\G)}(\alpha,a,b)$ with $\mu:= (c+d)/2$ and $\alpha:=(c-d)/2$. 
\end{itemize}
\end{rem}

For suitable choices of the parameters $\alpha, a$ and $b$, the operator $H_0^{(\G)}$ is of interest in quantum mechanics,  as shown by the examples below.

\begin{exam}\label{ExDirac} Let $m \geq 0$. The free one-dimensional discrete Dirac operator with mass $m$ on $\H_\G$ is defined by $D_m^{(\G)} := H_0^{(\G)}(m,1,-1).$
\end{exam}

\begin{exam}\label{ExampleTwoJacobi} 
If $\alpha=0$, $a,b>0$ and $a \neq b$, then the operator $H_0^{(\Z)}(0,a,b)$ is unitarily equivalent to the Su-Schriefer-Heeger (SSH) model introduced in \cite{SSH79}. The latter is a periodic Jacobi operator with period two on $\ell^2(\Z)$. Specifically, using the notation from \eqref{TwoPeriodic} below,
the SSH model corresponds to the operator $J_1^{(\Z)}(a,b, 0, 0)$.
\end{exam}
 
In the literature, the operators $H_0^{(\Z)}$ and $H_0^{(\Z_+)}$ are known respectively as the Laurent and Toeplitz operators associated with the symbol 
\begin{equation}\label{symbol}
    h(\theta): = \begin{pmatrix}
        \alpha & \bar{a}+\bar{b}e^{-i \theta} \\
{a}+{b}e^{i \theta} & -\alpha
    \end{pmatrix}, \quad \theta \in \T.
\end{equation}
The connection between $H_0^{(\Z)}$ and the symbol $h$  is performed by the Fourier transform, specifically the unitary operator ${\mathcal F}: \H_{\Z} \rightarrow L^2 ({\mathbb T}; \C^2 )$ defined by
\begin{equation}\label{FourierT}
({\mathcal F}u) (\theta ) := \sum_{n\in {\mathbb Z}} u(n) e^{in \theta}, \quad u \in \H_{\Z}, \ \theta \in  {\mathbb T}.
\end{equation}
Indeed, $H_0^{(\Z)} = \F^{-1} \widehat{H}_0 \F$, where $\widehat{H}_0$ is the multiplication operator on $ L^2 ({\mathbb T} ; \C^2 ) $ given by
\begin{equation}\label{MultOperator}
    \left(\widehat{H}_0f\right)( \theta) = h(\theta) f(\theta), \quad f \in  L^2 ({\mathbb T} ; \C^2 ), \ \theta \in \T.
\end{equation}
Let $\varphi_1,  \varphi_2 \in (-\pi, \pi]$ be such that  $a=|a|e^{i \varphi_1}$ and $b= |b|e^{i \varphi_2}$. We set $\varphi := \varphi_2 - \varphi_1$. For a fixed value of $\theta$, the eigenvalues of $ h(\theta)$ are given by $\pm \lambda(\theta)$, where $\lambda(\theta)>0 $ and
\begin{equation}\label{EigvSymbol}
    \lambda^2(\theta) = \alpha^2+ |a|^2 + |b|^2 +2|a||b| \cos(\theta + \varphi), \quad \theta \in \T.
\end{equation}
As  Proposition \ref{Prop1} below suggests, the nature of the spectrum of $H_0^{(\G)}$ is related to the properties of the function $\lambda$, which is non-constant because $a$ and $b$ are nonzero. One sees that $\pm \lambda(\T) = \con{I_{\pm}}$, where 
\begin{equation}\label{leftband}
     I_{\pm} : = \pm \left( \sqrt{\alpha^2+(|a|-|b|)^2}, \sqrt{\alpha^2+(|a|+|b|)^2}  \right).
\end{equation}
Furthermore, the set $ \kappa_1(\lambda):= \{ \lambda(\theta) : \theta \in \T, \lambda'(\theta)=0 \mbox{ or } \lambda'(\theta) \mbox{ does not exist}\}$ is finite and
\begin{equation}\label{kappa1}
    \kappa_1(\lambda) =  \left\{ \pm \sqrt{\alpha^2+(|a|-|b|)^2}, \pm \sqrt{\alpha^2+(|a|+|b|)^2} \right\}.
\end{equation}

Actually, $\lambda$ is not differentiable at $\theta \in \T$, (respectively, $\lambda'(\theta)=0$), if and only if $ \lambda(\theta)=0$ (resp. $\lambda$ attains  a nonzero extremum at $\theta$); see Figure \ref{fig1} for an illustration.

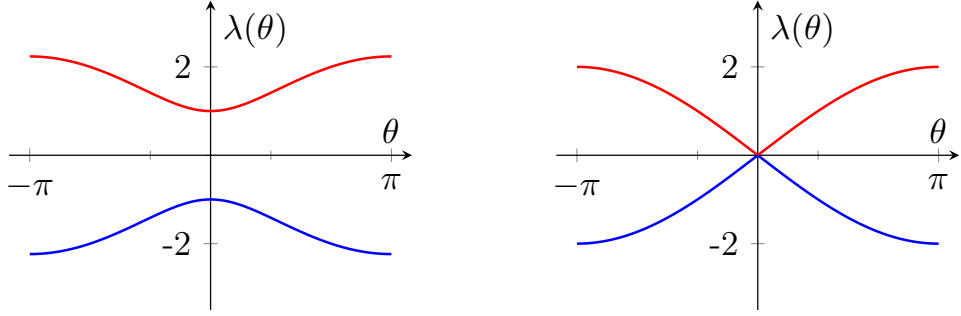
\begin{figure}[h]
    \centering
   	\begin{tikzpicture}[scale=1.2]
		
		\begin{axis}[
			width=6cm, height=5cm,
			xmin=-3.5, xmax=3.5,
			ymin=-3.5, ymax=3.5,
			axis lines=middle,
			xlabel={$\theta$}, ylabel={$\lambda(\theta)$},
			xtick={-3.1416,3.1416}, 
			xticklabels={$-\pi$,$\pi$},
			ytick={-2,2},           
			yticklabels={-2,2},
			minor x tick num=2,      
			minor y tick num=1,      
			grid=none                 
			]
			
			\addplot[red, thick, domain=-pi:pi, samples=200] 
			{sqrt(1^2 + 1^2 + 1^2 + 2*1*1*cos(deg(x + pi)))};
			
			\addplot[blue, thick, domain=-pi:pi, samples=200] 
			{-sqrt(1^2 + 1^2 + 1^2 + 2*1*1*cos(deg(x + pi)))};
			
		\end{axis}
		
		\begin{axis}[
			width=6cm, height=5cm,
			xmin=-3.5, xmax=3.5,
			ymin=-3.5, ymax=3.5,
			axis lines=middle,
			xlabel={$\theta$}, ylabel={$\lambda(\theta)$},
			xtick={-3.1416,3.1416}, 
			xticklabels={$-\pi$,$\pi$},
			ytick={-2,2},           
			yticklabels={-2,2},
			minor x tick num=2,
			minor y tick num=1,
			xshift=6cm,
			grid=none
			]
			
			\addplot[red, thick, domain=-pi:pi, samples=200] 
			{sqrt(0^2 + 1^2 + 1^2 + 2*1*1*cos(deg(x + pi)))};
			
			\addplot[blue, thick, domain=-pi:pi, samples=200] 
			{-sqrt(0^2 + 1^2 + 1^2 + 2*1*1*cos(deg(x + pi)))};
			
		\end{axis}
		
	\end{tikzpicture}

    \caption{Plot of $\theta \mapsto \lambda(\theta)$  (in red) and $\theta \mapsto -\lambda(\theta)$ (in blue) for $m=a=-b=1$ (left) and $m=0$,  $a=-b=1$ (right).}
    \label{fig1}
\end{figure}

\begin{prop}\label{Prop1} Let $ I_{\pm} $ be the intervals defined in \eqref{leftband}. Then
\begin{itemize}
\item[(i)] $ \sigma_{\ess}\left(H_0^{(\G)}\right) = \con{I_{-}} \cup \con{I_{+}}.$ There is a gap between the spectral bands $\con{I_{-}}$ and $\con{I_{+}}$ if and only if $\alpha \neq 0$ or $|a|\neq |b|$.
    \item[(ii)]  $H_0^{(\Z)}$ has purely absolutely continuous spectrum and $\sigma \left(H_0^{(\Z)} \right)= \ac \left(H_0^{(\Z)} \right)=\sigma_{\ess}\left(H_0^{(\Z)} \right) $.
    \item[(iii)] $H_0^{(\Z_+)}$ has no singular continuous spectrum. Furthermore, $\E_p\left(H_0^{(\Z_+)} \right)= \{-\alpha\}$ if   $|b|>|a|>0$, otherwise this set is empty.
\end{itemize}
\end{prop}

\begin{proof} We first prove (i) for $\G= \Z$ and (ii). Since $H_0^{(\Z)}$ is unitarily equivalent to the multiplication operator $\widehat{H}_0$ defined in \eqref{MultOperator}, and $\kappa_1(\lambda)$ is finite, the spectrum of $H_0^{(\Z)} $ is purely absolutely continuous and
   $$\sigma\left(H_0^{(\Z)}\right)= \ac\left(H_0^{(\Z)}\right)=\sigma_{\ess}\left(H_0^{(\Z)} \right) = -\lambda(\T) \cup \lambda(\T) = \con{I_{-}} \cup \con{I_{+}}.$$
Now, by Remark \ref{rmkH1} and Lemma \ref{FuncH0} below, $H_0^{(\Z_-)} \oplus H_0^{(\Z_+)}$ differs from $H_0^{(\Z)}$ by a finite-rank operator. Thus, the proof of (i) for $\G = \Z_{+}$ follows from the Weyl criterion.

Finally, we prove (iii). By a transfer matrix argument, we can show that $ H_0^{(\Z_+)}$ admits eigenvalues if and only if $|b|>|a|>0$, and in this situation the only eigenvalue is $-\alpha$ (see also \cite{P.C2}). The absence of singular continuous spectrum of $H_0^{(\Z_+)}$ follows from  Theorems \ref{main1} and \ref{main2}  below (see also \cite{rod}). This concludes the proof.
\end{proof}

\begin{rem}
In the SSH model, there is always a gap between $I_{-}$ and $I_{+}$. For the discrete one-dimensional  Dirac operator, a spectral gap exists as long as $m > 0$.
\end{rem}

In Section \ref{MainTheorems}, we present our main results regarding the spectral properties of the perturbed operator $ H^{(\G)}:=H_0^{(\G)}+V$, under the assumption that $V$ satisfies conditions related to the subspaces of bounded matrix-valued sequences on $\G$ introduced in the next section.

\subsection{Some subspaces of \texorpdfstring{$\ell^\infty(\G; M_d(\C))$}{l-infty(G; M	extunderscore d(C))}}\label{Perturbations}

In the sequel, for $d \in \{1,2\}$ we let $\| \cdot \|$ denote a matrix norm on the $\C$-vector space $M_d(\C)$, and we identify any  sequence $W: \G \to  M_d(\C)$ with the multiplication operator on $\ell^2(\G; M_d(\C))$ defined by
\begin{align*}
    \D(W) &:= \left\{ u \in \ell^2(\G; M_d(\C)) : (W(n)u(n))_{n \in \G} \in \ell^2(\G; M_d(\C)) \right\}\\
    (Wu)(n)&:=W(n)u(n), \quad u \in  \D(W).
\end{align*}
Furthermore, for $p \in \Z_+ \setminus \{0\}$, we define the sequences $\tau^p W$ and $\tau^{-p} W$ by
$$
(\tau^p W)(n) := \left\{ 
\begin{array}{ll}
0     & \mbox{if } \G= \Z_{+} \mbox{ and } n<p,\\
W(n-p) & \mbox{otherwise},
\end{array}
\right.
$$
and $(\tau^{-p} W)(n) := W(n+p)$, for $ n \in \G$. If $W \in  \ell^\infty(\G;M_2(\C))$, we write $W= (W^{ij})_{i,j=1,2}$, and $W^*$  is the sequence $((W(n))^*)_{n \in \G}$.

Now, let $k \in \Z_+ \setminus \{0\}$ and $m \in \{1,2\}$. We consider the seminorms on $\ell^\infty(\G;M_2(\C))$ given by
\begin{align*}
    q_{k,1} (W) &:= \sup_{n \in \G} \|n ( W - \tau^k W)(n) \|, \\
    q_{k,2} (W) &:= \sup_{n \in \G} \|n^2 ( W - 2\tau^k W + \tau^{2k} W)(n) \|.
\end{align*} 
If $W \in \ell^\infty(\G;M_2(\C))$, we let $q_0( W) \in \ell^\infty(\G)$ be the sequence
\begin{equation}\label{seqW0}
    { q_0(W)}(n)  = \|W^{12}(n)\| + \|W^{21}(n) \| + \|  (W^{11}-W^{22})(n)\| + \| (\tau W^{22}-W^{11})(n)\|.
\end{equation} 
Within this framework, we define five linear subspaces of   $\ell^\infty\left(\G; M_2(\C)\right)$, starting with 
\begin{align}
\Qq_{0,1}(\G)&:= \left\{ W \in \ell^\infty\left(\G; M_2(\C)\right):  \sup_{n \in \G} |n|(q_0(W))(n)< \infty \right\}, \label{Q0set} \\
\Qq_{k,m}(\G)&:= \left\{ W \in \ell^\infty(\G; M_2(\C)): \sum_{j=1}^m q_{k,j}(W)< \infty  \right\}.\label{Qkset}
\end{align}
Next, for $0<\beta <\gamma <  \infty$ fixed, we also define 
\begin{align}
   \S(\G) & := \left\{ W \in \ell^\infty(\G; M_2(\C)):  \int_{1}^\infty \sup_{ \beta r < |n| <  \gamma r} \| W(n)\| \, dr < \infty  \right\}, \label{short}\\
    \M_0(\G) &:=  \left\{ W \in \ell^\infty(\G; M_2(\C)):  \int_{\frac{1}{\gamma - \beta}}^\infty \sup_{ \beta r< |n| < \gamma r} q_0(W)(n) \, dr < \infty\right\}, \label{medium0} \\
    \M_{k}(\G) & :=  \left\{ W \in \ell^\infty(\G; M_2(\C)):  \int_{1}^\infty \sup_{ \beta r< |n| < \gamma r} \| (\tau^k W-W)(n)\|  \, dr < \infty\right\}. \label{mediumk} 
\end{align}
Here, we adopt the convention that the supremum over the empty set is zero. 

\begin{rem}
    By Lemma~\ref{AppendixLemma}, the definition of the sets $\S(\G) $ and $ \M_{k}(\G)$, for $k \geq 0$, does not depend on the specific choice of  $\beta$ and $\gamma$.
\end{rem}

\subsection{Main results}\label{MainTheorems}
In this section, we specify the assumptions on the perturbation $V$ of $H_0^{(\G)}$. In our framework, we always assume that $V  \in \B(\H_\G)$ has the form
\begin{equation}\label{GeneralPotential}
V := V_0+ \sum_{j=1}^N \left(S^jV_j + V_j^*S^{*j}\right),  
\end{equation}
with $ V_j \in \ell^\infty(\G ;M_2(\C))$, $V_0 = V_0^*$ and $\lim_{|n| \to \infty} \| V_j(n)\|=0$ for $j=0,1, \ldots, N$. This means that $V$ is self-adjoint and $V \in \K(\H_{\G})$. 

With the notation introduced in Section \ref{Perturbations}, we formulate a first assumption on $V$.

\begin{ass}\label{A1}  $V \in \K(\H_\G)$ is a self-adjoint operator as in \eqref{GeneralPotential}, and  there exists $k \in \Z_{+}$, with $k>0$, such that  
    $V_j \in \S(\G) + \M_k(\G)  + \Qq_{k,2}(\G)$  for each $j=0,1, \ldots, N.$
\end{ass}

For the remainder of this paper, we let $X^{(\G)}$ be the position operator  on $\H_{\G}$ defined by
\begin{equation}\label{PositionOpt}
    \left(X^{(\G)}u \right)(n) := nu(n) \mbox{ for } u \in \D\left(X^{(\G)}\right) \mbox{ and } n \in \G,
\end{equation}
where $ \D\left(X^{(\G)}\right):= \left\{ u \in \H_\G : (n u(n))_{n \in \G} \in \H_\G \right\}.$ Furthermore, for each $s \in \R$, we write $\mathscr{H}_s\left(X^{(\G)} \right)$ for the Sobolev space associated with
$X^{(\G)}$, so that  $\mathscr{H} := \H_\G$. Our first result reads as follows.

\begin{theo}\label{main1} Suppose that $V$ satisfies Assumption \ref{A1} for some $k>0$. Let $H^{(\G)}:=H_0^{(\G)}+V$ and consider the sets 
\begin{align}\label{mukThm2.1}
    \mu_k\left(H_0^{(\G)}\right)&:=(I_+ \cup I_{-}) \Big \backslash \kappa_k\left(H_0^{(\G)}\right), \\
    \tilde{\kappa}_k\left(H_0^{(\G), V}\right) &:= \kappa_k\left(H_0^{(\G)}\right) \cup \E_p\left( H_0^{(\G)}\right) \cup \E_p\left(H^{(\G)}\right),
\end{align}
  where
\begin{equation}\label{CriPointsAk}
     \kappa_k\left(H_0^{(\G)}\right):= \left\{ \pm 
     \sqrt{\alpha^2+|a|^2+|b|^2+2|a| |b| \cos\left( \frac{\pi j}{k} \right)} : j=0,1,\ldots,k\right\}.
\end{equation}
Then $ \sigma_{\ess} \left(H^{(\G)}\right) = \sigma_{\ess} \left(H_0^{(\G)}\right)$ and the following assertions hold:
\begin{enumerate}
    \item[1.] For every compact subset $I$ of $\R$ with $I  \subseteq \mu_k\left(H_0^{(\G)}\right)$, the set  
     $\E_p\left(H^{(\G)}\right) \cap I$ is finite, and each of these eigenvalues has finite multiplicity. The possible accumulation points of $ \E_p\left(H^{(\G)}\right)$ belong to $ \kappa_k\left(H_0^{(\G)}\right)$. 
    \item[2.]  For all  $s>1/2$ and $\K := \mathscr{H}_s \left(X^{(\G)} \right)$ the following LAP is satisfied: the holomorphic map $ \C_{\pm} \ni z \mapsto \left(H^{(\G)}-z\right)^{-1} \in \B\left( \K, \K^* \right)$ extends to a weak$^*$-continuous map on  $\R \setminus \tilde{\kappa}_k\left(H^{(\G)}\right)$. In particular,  $\sic\left(H^{(\G)}\right) = \emptyset$.
\end{enumerate}

\end{theo}

The following example provides concrete long-range perturbations satisfying the hypotheses of Theorem \ref{main1}.  Further examples can be constructed based on \cite[Section 10]{g-m-2}. 
 
\begin{exam}\label{ExLongRange}
 Let $x \in \R$ and $\bra x \ket := \sqrt{x^2+1}$. For any integer $p \geq 2$ and $r \in \R$, let $\ln_p^r(x):= (\ln_p(x))^r$, where $\ln_p(x)$ is defined recursively by
 $$ \ln_0(x):=1 , \ \ln_1(x):= \ln(1+x),  \  \ln_p(x):= \ln(1+ \ln_{p-1}(x)).$$
In addition, we define
       \begin{equation}\label{omega}
           \omega_{l}^{ r }(x) := \ln_{l+1}^r(\bra x \ket ) \displaystyle\prod_{p=0}^l \ln_{p}(\bra x \ket ), \quad  l \in \Z_+,  \ r, x \in \R.
       \end{equation}
Suppose that  $k \in \Z_{+} \setminus \{0\}$ and $V \in \K\left(\H_\G\right)$ is a self-adjoint operator as in \eqref{GeneralPotential}. If for all $j=0,1,\ldots, N$, there exist $l \in \Z_{+}$ and $ r>1$ such that 
 $(V_j- \tau^k V_j)(n)$ (respectively, $V_j(n)$) is $ O(|n|^{-1} (\omega^{r}_l(n))^{-1})$ as $|n| \to \infty$, then $V \in \M_k(\G)$ (respectively, $V \in \S(\G)$).
\end{exam}

In the gapless case, that is, when $\alpha=0$ and $|a|=|b|>0$, we have $\con{I}_{-} \cap \con{I}_{+} = \{ 0\}$ and $0$ is known as a conical point of the function $\lambda$ defined in \eqref{EigvSymbol} (see Figure \ref{fig1}). Furthermore, for each $k>0$ it holds $0 \in \kappa_k\left(H_0^{(\G)}\right)$, so that $0$ may be an accumulation point of $ \E_p\left(H_0^{(\G)} +V\right)$ for some perturbations $V$ satisfying Assumption \ref{A1}. We can avoid this kind of situation by introducing an alternative hypothesis on $V$.
 
\begin{ass}\label{A2}
  $V \in \K(H_\G)$ is a self-adjoint operator as in \eqref{GeneralPotential} and for each $j=0,1, \ldots, N,$ it is true that $V_j \in \S(\G) + \M_0(\G)$. 
\end{ass}

This leads us to our second main result. 

\begin{theo}\label{main2}
    Suppose that  $\alpha=0$, $|a| = |b|$ and $V$ satisfies Assumption \ref{A2}. Consider the operator $H^{(\G)}:=H_0^{(\G)}+V$ and the sets 
    \begin{align}
    \mu_0\left(H_0^{(\G)}\right)&:= (-2|a|,2|a|), \\
     \kappa_0\left(H_0^{(\G)}\right) &:= \{\pm 2|a|\},\\
    \tilde{\kappa}_0\left(H_0^{(\G)}, V \right) &:= \kappa_0\left(H_0^{(\G)}\right) \cup \E_p\left(H^{(\G)}\right).
    \end{align}
  Then $ \sigma_{\ess} \left(H^{(\G)}\right) = \sigma_{\ess} \left(H_0^{(\G)}\right)$ and the following assertions hold:

\begin{enumerate}
    \item[1.] For every compact subset $I$ of $\R$ with $I  \subseteq \mu_k\left(H_0^{(\G)}\right)$, the set $\E_p\left(H^{(\G)}\right) \cap I$ is finite, and each of these eigenvalues has finite multiplicity. The possible accumulation points of $ \E_p\left(H^{(\G)}\right)$ belong to $ \kappa_0\left(H_0^{(\G)}\right)$.
    \item[2.] For all  $s>1/2$ and $\K := \mathscr{H}_s \left(X^{(\G)} \right)$ the following LAP is satisfied: the holomorphic map $ \C_{\pm} \ni z \mapsto \left(H^{(\G)}-z\right)^{-1} \in \B\left( \K, \K^* \right)$ extends to a weak$^*$-continuous map on   $ \R \setminus \tilde{\kappa}_0\left(H^{(\G)}\right)$. In particular,  $\sic\left(H^{(\G)}\right) = \emptyset$.
\end{enumerate}  
\end{theo}

\begin{rem}\label{RmkLAP1}

In  Theorems \ref{main1} and  \ref{main2} the existence of a LAP is equivalent to the fact that for any $f,g \in \K:= \mathscr{H}_s \left(X^{(\G)} \right) $, the limits 
    $$\innp{f}{\left(H^{(\G)}-x \pm i 0 \right)^{-1} g}:=  \lim_{\eps \to 0^+} \innp{f}{\left(H^{(\G)}-x \pm i \eps \right)^{-1}g} $$ 
    exist locally uniformly in  $ \R  \backslash \tilde{\kappa}_k\left(H^{(\G)}\right)$ for $k>0$ and $k=0$, respectively. In particular, the limit functions are continuous and for $s > 1/2$ and  each compact set $I \subseteq \R \backslash \tilde{\kappa}_k\left(H^{(\G)}\right)$ it holds 
    $$ 
    \sup_{\Rp(z) \in I, \, \Ip(z) \neq 0} \, \left\| \left\bra X^{(\G)}\right\ket^{-s} \left(H^{(\G)}-z \right)^{-1} \left\bra X^{(\G)}\right\ket^{-s} \right\| < \infty.
    $$
    We point out that under hypothesis of Theorem \ref{main1} or \ref{main2} a LAP also holds in the Banach space $\left(D\left(X^{(\G)}\right),\H_{\G} \right)_{1/2,1}$,  obtained via real interpolation \cite{abmg}.
\end{rem}

\begin{rem}
    Consider the case in which $H_0^{(\G)}$ has no spectral gap, and let $V \in \K\left(H_\G \right)$ be a self-adjoint operator as in \eqref{GeneralPotential}. If   $V_j \in \S(\G)$ for all $j=0,1, \ldots,N$, both Theorems \ref{main1} and \ref{main2} apply, but Theorem \ref{main2} says that $0$ can not be an accumulation point of $ \E_p\left(H_0^{(\G)} +V\right)$. Observe also that in this situation, $\M_0(\G)$  is strictly contained in  $ \M_1(\G)$.  
\end{rem}

\subsection{Connection with previous works}\label{connection}

This section aims to contrast our main theorems with previous results from \cite{P.C-sah, sahahah} on block Jacobi operators, and from \cite{g-h, kopylova, k-t2,  k-t} on the one-dimensional discrete Dirac operator.

\subsubsection{Block Jacobi operators}
In order to compare our results with \cite{P.C-sah, sahahah}, we first recall the definition of Jacobi operators acting on $\H_\G$.

\begin{defi} Let $d \in \Z_{+} \setminus \{0\}$, and  $(A_n)_{n \in \G}$ and $(B_n)_{n \in \G}$ be two bounded sequences in $M_2(\C^d)$ such that $B_n= B_n^*$ for all $n \in \G$. If $d=1$, we also assume that $A_n>0$ for all $n \in \G$. The Jacobi operator $J_d^{(\G)}= J_d^{(\G)}\left((A_n)_{n \in \G}, (B_n)_{n \in \G} \right)$ associated with these sequences acts in $\ell^2(\Z; \C^d)$ by
\begin{equation}\label{TwoPeriodic}
    \left( J_d^{(\G)} u\right)(n) := \left\{ \begin{array}{ll}
       B_n u(n) + A_n u(n+1)   & \mbox{if } \G= \Z_+ \mbox{ and } n=0,  \\
       A_{n-1}^* u(n-1) + B_n u(n) + A_nu(n+1)  & \mbox{otherwise},
    \end{array}\right.
\end{equation}
for $u \in \ell^2(\Z; \C^d)$ and $n \in \G$. Furthermore,
\begin{itemize}
    \item[(i)] Given $A, B \in M_d(\C)$, we define  $J_d^{(\G)}(A,B) := J_d^{(\G)}((A)_{n \in \N},(B)_{n \in \N})$.
  \item[(ii)]  If $d=1$, and there is $N \in \Z_{+}$ with $N>0$ such that $A_{n+N}= A_n$ and $B_{n+N}= B_n$ for all $n \in \G$, we say that $J_1^{(\G)}$ is  a periodic Jacobi operator on $\ell^2(\G)$ with period $N$.
\end{itemize}
\end{defi}

Note that the map $((A_n)_{n \in \N},(B_n)_{n \in \N}) \mapsto J_d^{(\G)}((A_n)_{n \in \N},(B_n)_{n \in \N}) \in \B(\ell^2(\G; \C^2))$ is linear. In what follows, we focus on the case $d=2$ and assume that there exist $A, B \in M_2(\C)$ such that 
\begin{equation} \label{ConvergenceAnBn}
    \lim_{|n| \to \infty} \|A_n-A \| + \| B_n-B\|=0.
\end{equation}
Then  $J_2^{(\G)} = J_2^{(\G)}(A,B) +V,$ where  $V:= J_2^{(\G)}\left((A_n-A)_{n \in \G}, (B_n-B)_{n \in \G} \right)$ is a compact self-adjoint operator on $\H_{\G}$. The nature of $\sigma \left(J_2^{(\G)} \right)$ depends on the rate of convergence in \eqref{ConvergenceAnBn}. For example, the following two assertions were shown in \cite{sahahah} (respectively, \cite{ P.C-sah}) when $\G=\Z$ (resp. $\G=\Z_+$):

\begin{enumerate}
    \item Let $\lambda_1, \lambda_2: \T \to \R$ be analytic functions such that for all $\theta \in \T$, $\{\lambda_1(\theta), \lambda_2(\theta) \}$ are the repeated eigenvalues of the symbol $ e^{-i \theta}A  + B+ A^*e^{-i \theta}$. Define the set 
    \begin{equation}\label{criticalSah}
        \kappa \left( J_2^{(\G)}(A,B) \right):= \{ \lambda_1(\theta) : \lambda_1'(\theta)=0 \} \cup \{ \lambda_2(\theta) : \lambda_2'(\theta)=0 \}.
    \end{equation}
    Suppose that
     \begin{equation}\label{limm0}
    \lim_{|n| \to \infty} |n|(\|A_{n}-A \| + \|B_{n}-B\|)=0.
\end{equation}
    Then all the eigenvalues of $J_2^{(\G)}$ outside $\kappa\left(J_2^{(\G)}(A,B) \right)$ have finite multiplicity, and their possible accumulation points are contained in $\kappa\left(J_2^{(\G)}(A,B)\right)$.

\item The conclusions of Theorem \ref{main1} remain true for $J_2^{(\G)}$ and $\kappa \left( J_2^{(\G)}(A,B) \right)$ instead of $ H^{(\G)}$ and $\kappa_k\left(H_0^{(\G)}\right)$, respectively, provided that
\begin{equation}\label{limm1}
    \int_1^\infty  \sup_{r < |n| < 2r}(\|A_{n}-A \| + \|B_{n}-B\|)  \, dr < \infty.
\end{equation}
\end{enumerate}

We now assume further that there are $\alpha \in \R$ and $a,b \in \C \setminus \{0\}$ such that
\begin{equation}\label{eqChoiceAandB}
    A := \begin{pmatrix}
        0 & \con{b} \\
        0 & 0
    \end{pmatrix} \quad \mbox{and}  \quad B := \begin{pmatrix}
        \alpha & \con{a}\\
        a & -\alpha
    \end{pmatrix}.
\end{equation}
Then  $ J_2^{(\G)}(A,B)= H_0^{(\G)}(\alpha, a,b)$, where $H_0^{(\G)}$ is given by \eqref{DefH0}. Furthermore, for the unitary operator  $U: \ell^2(\G) \to  \ell^2(\G; \C^2)$ defined by  \begin{equation}\label{EqUnitaryJacobi}
    (U\psi)(n) = \begin{pmatrix}
    \psi(2n+1) \\
    \psi(2n)
\end{pmatrix}, \quad \psi \in  \ell^2(\G), \, n \in \G,
\end{equation}
we have that
\begin{equation}
     U^{-1} H_0^{(\G)}(\alpha, a,b)U =: J_1^{(\G)}((a_n)_{n \in \G}, (b_n)_{n \in \G})
\end{equation}
is the periodic Jacobi operator with period $2$ associated with the sequences $(a_n)_{n \in \G} $ and $(b_n)_{n \in \G}$ given by 
$$a_{2n}:= a, \, a_{2n+1}:= b, \, b_{2n}:= -\alpha, \, b_{2n+1}:=\alpha   , \quad n \in \G. $$
Observe that the sets  $\kappa\left(J_2^{(\G)}(A,B) \right)$ and  $\kappa_1\left(H_0^{(\G)}\right)$ given in \eqref{criticalSah}  and \eqref{kappa1}, respectively, are the same provided that $ \alpha \neq 0$ or $|a| \neq |b|$.  In this case, our main results allow us to treat the larger class of perturbations $V$ such that 
\begin{equation}\label{Lim1}
    \lim_{|n| \to \infty} |n|(\|A_{n+k}-A_{n} \| + \|B_{n+k}-B_{n}\|)=0,
\end{equation}
and
\begin{equation}\label{limm11}
    \int_1^\infty  \sup_{r < |n| < 2r}(\|A_{n+k}-A_n \| + \|B_{n+k}-B_n\|)  \, dr < \infty,
\end{equation}
for some $k > 0$. Indeed, with the notation introduced in Section \ref{MainTheorems}, \eqref{Lim1} and \eqref{limm11} means that the components of $V$ belong to $ \Qq_{k,1}(\G)$  and $\M_k(\G)$, respectively. Thus, Assumption~\ref{A1} is fulfilled by $V$ and Theorem \ref{main1} holds for $J^{(\G)}$ instead of $ H^{(\G)}$. 
We also point out that  $V \in \Qq_{k,2}(\G)$ provided that
\begin{equation}\label{Lim2}
    \lim_{|n| \to \infty} n^2(\|A_{n}-2A_{n+k}+ A_{n+2k}\| + \|B_{n}-2B_{n+k}+ B_{n+2k}\|)=0.
\end{equation}

\begin{rem} Condition \eqref{limm0} implies \eqref{Lim1}, but the converse does not hold. To see this,  consider  $A_n= A + (\omega_0^1(n))^{-1}$ for $n \in \G$ and $\omega_0^1$ defined as in \eqref{omega}. From this basic example, one can also construct potentials $V$ for which \eqref{Lim1} is true for some $k>1$ but fails for $k=1$. 
\end{rem}

\begin{rem}\label{remDiscreteLaplacian}
    If $|a|=|b|$ and $\alpha=0$, we have that $H_0^{(\Z)}(\alpha, a,b)$ is also unitarily equivalent to the scalar periodic Jacobi operator $2 |a| \Delta^{(\Z)}$, where $\Delta^{(\Z)}$ is the discrete Laplacian on $\ell^2(\Z)$ defined by
    $$ (\Delta^{(\Z)} \psi)(n)= \psi(n-1)+\psi(n+1), \quad  \psi \in \ell^2(\Z), \, n \in \N.$$
For a detailed analysis of certain long-range perturbations of this operator, we refer to \cite{g-m-2}. An alternative approach is discussed in Section \ref{A0Gapless}, which serves as the starting point for the proof of  Theorem \ref{main2}.
\end{rem}

\subsubsection{One-dimensional discrete Dirac operators}
As mentioned before, limiting absorption principles for one-dimensional discrete Dirac operators have been established in \cite{g-h, kopylova, k-t2,  k-t} for the massive case. Recall that the corresponding unperturbed operator $D_m^{(\G)}$ was introduced in Example \ref{ExDirac}. The first advantage of our work is that some massless discrete Dirac operators are covered by Theorem~\ref{main2}. In this section, we assume that $m>0$ and exhibit examples of perturbed operators $D_m^{(\Z)}+V$ that are not included in the mentioned references, but for which the conclusions of Theorem~\ref{main1} are true if we take $D_m^{(\Z)}+V$ instead of $H^{(\Z)}$.

It is known that the spectrum of $D_m^{(\G)}$ is purely absolutely continuous (see Proposition \ref{Prop1} or \cite{c-d-p}). Furthermore,  $\sigma\left(D_m^{(\G)} \right) = \overline{\Gamma}$, where 
\begin{equation}\label{gamma}
    \Gamma:= \left(-\sqrt{m^2+4}, -m \right) \cup \left(m, \sqrt{m^2+4} \right).
\end{equation}
In \cite{kopylova} it is assumed that the matrix-valued potential $V:=V_0 \in \ell^\infty(\Z ;M_2(\C))$ satisfies 
\begin{equation}\label{EEE}
   V_0^{12}(n) \neq -1 \quad \mbox{and} \quad |V_0^{ij}(n)| \leq  C (1+|n|)^{-\rho}, \quad i,j \in \{1,2\}, \ n \in \Z,
\end{equation}
for some $\rho>1$. Under these conditions, \cite[Theorem 4.1]{kopylova} states that $\sigma_{\ess}\left(D_m^{(\Z)}+V\right) = \con{\Gamma}$ and for all $s>1/2$ and $\K:= \mathscr{H}_s\left( X^{(\Z)} \right)$ the following LAP holds: the holomorphic map $ \C_{\pm} \ni z \mapsto \left(D_m^{(\Z)}+V-z\right)^{-1} \in \B\left( \K, \K^* \right)$ extends to a weak$^*$-continuous map on $\C_\pm  \cup \Gamma$. In particular,    there are no eigenvalues of the perturbed operator embedded in  $\Gamma$. The same kind of conclusions are obtained in \cite{k-t2,  k-t} by assuming that $ V:=V_0 \in \ell^\infty(\Z ;M_2(\R))$ is a real potential such that 
\begin{equation}\label{EEE2}
    V_0^{12}(n) \neq -1 \quad \mbox{and} \quad V_0^{ij} \in \ell^1(\Z), \quad    i,j \in \{1,2\}, n \in \Z .
\end{equation}
The next example gives a potential that satisfies 
Assumption \ref{A1}, but neither  \eqref{EEE} nor \eqref{EEE2}.

\begin{exam}\label{ExampleKopylov}
Let $V:=V_0 \in \ell^\infty(\Z ;M_2(\C))$ be the sequence defined by
$V_0 (n) := (\omega_0^1(n))^{-1} I_2$ for $n \in \Z$ , where $I_2 \in M_2(\C)$ is the identity matrix and $\omega_0^1$ is defined as in \eqref{omega}. Then $V  \in \K\left(\H_\Z\right) \cap \Qq_{1,2}(\Z)$. By Theorem \ref{main1},  $ D_m^{(\Z)} + V$ has no singular continuous spectrum, and the possible accumulation points of $\E_p\left( D_m^{(\Z)} + V \right)$ belong to  $\left\{ \pm m , \pm \sqrt{m^2+4} \right\} $. However, $V_0$ fails to meet $\eqref{EEE}$ for any $\rho>1$, and  \eqref{EEE2} and . 
\end{exam}

To facilitate comparison with \cite{g-h}, we specify two assumptions on $V$.

\begin{itemize}
    \item[(\bf A1)]  $V:=V_0 + SV_1+V_1^*S^* $ is a compact self-adjoint operator on the space $\H_{\G}$ such that $V_0, V_1 \in \ell^\infty(\G ;M_2(\C))$,  $ V_1^{11}= V_1^{12}=V_1^{22}=0 $ and   
\begin{equation*}
          V_0^{12}(n) \neq -1 \mbox{ and } V_1^{21}(n) \neq 1 \quad \mbox{for all } n \in \G.
    \end{equation*}

    \item[(\bf A2)] There exist $p_1, p_2 \in \Z_+ \setminus \{0\}$ such that 
    \begin{equation*}
        \at{V_0^{jj}}{\Z_+} - \tau^{p_1} \at{V_0^{jj}}{\Z_+} \in \ell^1(\G; \R) \mbox{ and } 
        \at{V_l^{21}}{\Z_+} - \tau^{p_2} \at{V_l^{21}}{\Z_+} \in \ell^1(\G; \C),
    \end{equation*}
    for $j=1,2$ and $l=0,1$.
\end{itemize}

\begin{rem} By Lemma \ref{AppendixLemma},  if $V_1 \in \M_{p_2}(\Z_{+})$ and $V_0 \in \M_{p}(\Z_{+})$ for some common factor $p$ of $p_1$ and $p_2$, then assumption $(\bf{A2})$ is fulfilled.
\end{rem}

According to  \cite[Theorems 3.1 and 3.3]{g-h}, if the assumptions $(\bf{A1})$ and $(\bf{A2})$ are true, then the spectrum of $D_m^{(\G)}+V$ is purely absolutely continuous on the set $\Gamma$ defined in \eqref{gamma}. Proposition~\ref{counterex} below shows the existence of potentials of the form indicated in $(\bf{A1})$ that do not satisfy   $(\bf{A2})$, but are nevertheless covered by our analysis. Before stating it, we introduce a definition and prove a preliminary lemma.

\begin{defi}
    Let $(\P_n)_{n \in \Z_{+}}$ be a partition of $\Z_{+}$ such that $|\P_{n}|<|\P_{n+1}|$ for all  $n \geq 1$, and 
    \begin{equation}\label{MinmaxPn}
        \alpha_n:= \min \P_n < \max \P_n =: \beta_n = \alpha_{n+1}-1, \quad n \in \Z_{+}.
    \end{equation}
    A sequence of functions $(f_n)_{n \in \Z_{+}}$ on $ \Z_{+}$ is called  subordinate to the partition $(\P_n)_{n \in \Z_{+}}$ if $f_n \geq 0$, $\supp(f_n) \subseteq \P_n$ and $\| f_n\|_1:=  \sum_{j \in \P_n} f_n(j) >0$
    for each $n \in \Z_{+}$.
\end{defi}

\begin{lem}\label{SpecialSequence}
Let $(f_n)_{n \in \Z_{+}}$ be a sequence of functions subordinate to a partition $(\P_n)_{n \in \Z_{+}}$ of $\Z_{+}$, and let $(a_n)_{n \in \Z_{+}}$ be a sequence of positive real numbers such that the series $\sum_{n=0}^\infty (-1)^na_n$ is conditionally convergent. For each $n \in \Z_{+}$ and $j \in \P_n$, define
\begin{equation}\label{bnsequence}
    b_j:= \frac{(-1)^n}{  \| f_n\|_1 }a_n f_n(j).
\end{equation}
Then 
\begin{itemize}
\item[\textnormal{(i)}] The series $\sum_{j=0}^\infty b_j$ converges.
\item[\textnormal{(ii)}] For any $p \in \Z_{+} \setminus \{ 0 \}$, the series $ \sum_{j=0}^\infty |b_j+b_{j+1}+ \cdots+ b_{j+p-1}| $ diverges.
    \end{itemize}
In addition, let  $\alpha_n$ and $\beta_n$ be as in \eqref{MinmaxPn} and  assume that $f_n(\alpha_n)=0$ for  all $n \in \Z_{+}$,
\begin{equation}\label{BoundL1}
   L_1:= \sup_{n \in \Z_+} \frac{\beta_n a_n\| f_n\|_{\infty}}{ \| f_n\|_1}  := \sup_{n \in \Z_+} \frac{\beta_n a_n}{ \| f_n\|_1} \left( \max_{j \in \P_n} f_n(j) \right) \in \R,
\end{equation}
and
\begin{equation}\label{BoundL2}
    L_2:= \sup_{n \in \Z_+}  \frac{\beta_n^2 a_n}{ \| f_n\|_1} \max M_n \in \R,
\end{equation}
where $M_n:= \left\{  | f_n(j)-f_n(j+1)| : \alpha_n \leq j < \beta_n \right\} \cup 
 \{f_n(\beta_n)\}$ for each $n \in \Z_{+}$. Then 
\begin{itemize}
     \item[\textnormal{(iii)}] The sequences $(j b_j)_{j \in \Z_{+}}$ and $(j^2(b_j -b_{j+1}))_{j \in \Z_{+}}$ are bounded.
\end{itemize}
\end{lem}

\begin{proof} We first show (i). Let $(S_n)_{n \in \N}$ be the sequence of partial sums of the series  $\sum_{j=0}^\infty b_j$. Observe that for every $n \in \Z_{+}$, we have
       $$
            S_{\beta_n}:=  \sum_{j=0}^{\beta_n} b_j= \sum_{k=0}^n \sum_{j \in \P_k} b_j  = \sum_{k=0}^n  (-1)^ka_k. 
       $$
    Thus, $\left(S_{\beta_n}\right)_{n \in \Z_{+}}$ converges to the same limit of the series $\sum_{n=0}^\infty (-1)^na_n$. Furthermore, for $ \beta_n < l < \beta_{n+1}$, we have $S_{\beta_n} \leq S_l \leq S_{\beta_{n+1}}$ when $n$ is odd, and $S_{\beta_n} \geq S_l \geq S_{\beta_{n+1}}$ when $n$ is even. This implies (i).
    
To show (ii), let $p \in  \Z_{+} \setminus \{0\} $. Then there is $n_p \in \Z_{+}$ such that $ p < |\P_{n_p}|$. Thus,  $ p < \beta_{n_p}$ and for each $N>{n_p}$ we obtain
        \begin{align*}
            \sum_{j=0}^{\beta_N} |b_j+b_{j+1} + \cdots + b_{j+p-1}| & \geq 
            \sum_{n=n_p}^{N} \sum_{j = \alpha_n}^{\beta_n-p+1} |b_j+b_{j+1} + \cdots +  b_{j+p-1}| \\
            & = \sum_{n=n_p}^{N} \sum_{j = \alpha_n}^{\beta_n-p+1} \frac{a_n}{\|f_n\|_1  }(f_n(j)+ \cdots + f_n(j+p-1)) \\
            & \geq \sum_{n=n_p}^{N} \frac{a_n}{\|f_n\|_1 } \|f_n\|_1 = \sum_{n= n_p}^{N} a_n.
        \end{align*}
From this, we deduce that the series given in (ii) diverges. 

It remains to prove (iii). For all $n \in \Z_+$ and $j \in \P_n$ we have
        \begin{align*}
            |jb_j| &= \frac{j a_n}{\|f_n\|_1 } f_n(j ) \leq  \frac{\beta_n a_n \| f_n\|_\infty}{\|f_n\|_1}  \leq  L_1.
        \end{align*}
This means that $(j b_j)_{j \in \Z_+}$ is bounded. To show  that the sequence $(j^2(b_j -b_{j+1}))_{j \in \Z_+}$ is also bounded, we consider two cases. If $j \in \P_n \setminus \{ \beta_n\}$, then
\begin{align*}
    |j^2(b_j-b_{j+1})| &\leq  \frac{\beta_n^2 a_n}{\|f_n\|_1} |f_n(j) -f_n(j+1)| \leq    L_2.
\end{align*}
While for $j= \beta_n$ we have
\begin{align*}
    |j^2(b_j-b_{j+1})| & = \frac{\beta_n^2a_n}{\|f_n\|_1}  f_n(\beta_n)  \leq  L_2.
\end{align*} 
 
This concludes the proof.
 \qedhere
\end{proof}
As an illustration, we provide a concrete example of a sequence of the type described in Lemma~\ref{SpecialSequence}.  
\begin{exam} Let $\P_0:=\{0,1\}$ and $\P_n:=\{ j \in \Z_{+}: 2^{n} \leq j < 2^{n+1}\}$ for each $n \geq 1$. We let $f_0(0):=0$, $f_0(1):=1$, and $f_0(j):=0$ if $j>1$.  For $n  \geq 1$, we define  $f_n: \Z_+ \to [0, \infty)$ by 
$$ f_n(j):= \left\{ \begin{array}{ll}
  \min \{ j -2^{n} , 2^{n+1}-j\}   &  \mbox{ if } j \in \P_n,\\
   0  & \mbox{ otherwise.} 
\end{array} \right.
$$
Then $\|f_n\|_1 = 2^{2(n-1)} $ and $\| f_n \|_\infty =2^{n-1}$ for $n \geq 1$, and the sequence $(f_n)_{n \in \Z_+}$ is subordinate to the partition $(\P_n)_{n \in \Z_+}$ of $ \Z_+$.  Furthermore,  for any sequence $(a_n)_{n \in \Z_+}$ of positive real numbers for which the series $\sum_{n=0}^\infty (-1)^na_n$ is conditionally convergent, the conditions \eqref{BoundL1} and \eqref{BoundL2} hold, and so the sequence $(b_j)_{j \in \Z_+}$ defined by \eqref{bnsequence} fulfills all assertions of Lemma~\ref{SpecialSequence}.
\end{exam}

We conclude this section with a proof of the existence of potentials that satisfy the hypotheses of Theorem~\ref{main1}, but do not fulfill the conditions introduced in \cite{goh}.

\begin{prop}\label{counterex}
    There exist $V \in \K(\H_\G)$ such that $(\bf{A1})$ and Assumption \ref{A1} are satisfied for $k=1$, but $(\bf{A2})$ is not fulfilled. 
\end{prop}

\begin{proof}
Let  $(b_j)_{j \in \Z_{+}}$ be any sequence of the form \eqref{bnsequence} for which all conclusions of Lemma \ref{SpecialSequence} hold. From item (i) of this lemma, we can define $V:= V_0 \in \B(\H_\G)$ by $V_0^{12} = V_0^{21} = V_0^{22} =0$, $V_0^{11}(n)=0$ if $n <0$, and 
\begin{equation*}
V_0^{11}(0) = - \sum_{j=0}^\infty b_j,  \quad V_0^{11}(n+1)- V_0^{11}(n)=b_n \quad \mbox{for} \quad n \in \Z_{+}.
\end{equation*}
According to item (iii) from Lemma \ref{SpecialSequence}, $V \in \K(\H_{\G}) \cap \Qq_{1,2}(\G)$ is self-adjoint and satisfies  $(\bf{A1})$, while item (ii) from this lemma shows that $(\bf{A2})$  is not fulfilled.
\end{proof}

\begin{rem} By Lemma \ref{AppendixLemma}, the potential $V \in \K(\H_\G)$ defined in the proof of Proposition~\ref{counterex} belongs to $\Qq_{1,2}(\G) \setminus \left( \S(\G) \cup \M_1(\G) \right)$. 
\end{rem}

\section{A brief review on positive commutators methods}\label{ReviewMourre}

This section introduces the key ideas from Mourre theory that will be used later. Most of the notations and definitions are adopted from \cite{abmg}, to which we refer for further details. Through this section, unless explicitly stated otherwise, let  $ I\subseteq \mathbb{R}$ denote a Borel set, $A$ a self-adjoint operator acting on a separable complex  Hilbert space $\Hh$, with domain $\D(A)$, and $H, T \in \B\left(\Hh\right)$ with $H$ self-adjoint.

The Mourre theory refers to the positive commutator method introduced in \cite{EM}. This theory enables the study of the nature of the spectrum of $H$ in $I$,  based on the regularity of $H$ with respect to $A$ and the local positivity of the commutator $[iA, H]$.   We begin by defining the regularity classes associated to $A$.

\begin{defi}\label{RegularityClasses} 
\begin{itemize}
    \item[(i)] Let $m \in \Z_+$. We say that $T$ is of class $C^m(A)$, and write $ T \in C^m(A)$, if the map $\Tau: \R \to \B\left(\Hh\right)$, defined by $\Tau(t):=e^{i A t} T e^{-iAt} $ for $t \in \R$, is of class $C^m(\R)$ with respect to strong topology of $\B\left(\Hh\right)$. The notation $T \in C^\infty(A)$ means that $T \in C^m(A) $ for all $m \geq 1$. 
    \item[(ii)] We say that $T$ is of class $\Cc^{1,1}(A)$ provided that
    $$\int_0^1 \| e^{i At}T e^{-i At}+e^{-i At}T e^{i At}-2T\| \, \frac{dt}{t^2}<\infty.$$
In this case, we write $ T \in \Cc^{1,1}(A)$.

\end{itemize}
\end{defi}

\begin{rem}\label{starAlgebras}
\begin{enumerate}
    \item[(i)] The operator $T$ is of class $C^1(A)$ if and only if $T$ leaves the domain of $A$ invariant,  and the operator $i(AT-TA)$ defined on $\D(A)$ has a bounded extension to $\Hh $. We denote this bounded extension by $[iA,T]$. In this case, $\Tau'(0) = [iA,T].$
    \item[(ii)] $T  \in C^2(A)$ is equivalent to the property that $T\in C^1(A)$ and $ [iA, T ] \in C^1(A)$, so that the second order commutator $ [iA, [iA, T]]$ has a bounded extension to the whole space. 
    \item[(iii)] The sets $C^m(A)$  and $\Cc^{1,1}(A)$ are Banach $*$-subalgebras of $\B\left(\Hh\right)$ and the inclusions $C^m(A) \subset  \Cc^{1,1}(A) \subset  C^1(A)$ are true for all $m \geq 2$.
\end{enumerate}
\end{rem}

From now on, we adopt the following notation: given two operators $R, T \in \B\left(\Hh \right)$, we write $R \simeq T$  to express that $R-T \in \K(\Hh )$, and $ R \gtrsim T$ means that $R \geq T + K$ for some $K \in \K(\Hh )$. 

\begin{defi}\label{DefMourre} Let $H \in C^1(A)$. The operator   {$A $ is conjugate to $H $ on $I $}   if there exist $ c >0 $  such that
\begin{equation}\label{ME}
    E_I(H)[iA,H]E_I(H) \gtrsim c E_I(H). 
\end{equation}
This is known as a  {Mourre estimate} for the triplet $(H, A, I)$. If 
\begin{equation}\label{SME}
    E_I(H)[iA,H]E_I(H) \geq c E_I(H),
\end{equation}
then {$A $ is strictly conjugate to $H $ on $I$}, and this inequality is called a   {strict Mourre estimate} for the triplet $(H, A, I)$. We also define the subsets $ \tilde{\mu}_{A}^{\pm}(H)$ and $ {\mu}_{A}^{\pm}(H)$ of $\R$  by
\begin{align*}
     \tilde{\mu}_{A}^{\pm}(H)&:= \{ x \in \sigma_{\ess}(H): \exists c>0 \mbox{ and an open interval } I ; x \in I\mbox{ and } \eqref{ME} \mbox{ holds with} \pm A \}, \\
{\mu}_{A}^{\pm}(H)&:= \{ x \in \sigma_{\ess}(H): \exists c>0 \mbox{ and an open interval } I; x \in I\mbox{ and } \eqref{SME} \mbox{ holds with} \pm A\}.
\end{align*}
We also let $\tilde{\mu}_A(H):= \tilde{\mu}_{A}^{+}(H) \cup \tilde{\mu}_{A}^{-}(H)$, $\mu_{A}(H):= \mu_{A}^{+}(H) \cup \mu_{A}^{-}(H)$, and
\begin{equation}
\kappa_A(H):= \sigma_{\ess}(H) \setminus  \tilde{\mu}_A(H).
\end{equation}
We say that $\kappa_A(H)$ is the set of critical points of $H$ with respect to $A.$
\end{defi}

In Definition \ref{DefMourre}, we have  $\pm A$ is locally conjugate  (respectively, strictly conjugate) to $H$ on $  \tilde{\mu}_{A}^{\pm}(H)$ (respectively, $ {\mu}_{A}^{\pm}(H)$).  The following result describes the difference between $\tilde{\mu}_A(H)$ and $\mu_{A}(H)$.
\begin{theo}\label{MainMourre1}

Let $H \in C^1(A)$. Then
\begin{itemize}
    \item[(i)] The spectrum of $H$ in $\mu_A(H) $ is purely continuous and  $ \tilde{\mu}_A(H) \setminus \mu_A(H)$ consists of eigenvalues of $H$ of finite multiplicity. 
    \item[(ii)] If  $\con{I} \subseteq \tilde{\mu}_A(H)$, the set $\E_p(H) \cap I$ is finite. In particular, the    
possible accumulation points of $ \E_p(H)$ belong to $ \kappa_A(H)$.
\end{itemize}
\end{theo}

\begin{proof}
    Assertion (i) is a consequence of the Virial Theorem \cite[Proposition 7.2.10]{abmg}. To show (ii), observe that if  $\con{I} \subseteq \tilde{\mu}_A(H)$, then $\con{I}$ is compact since $ \tilde{\mu}_A(H) \subseteq \sigma(H)$ is bounded. Hence,  taking a finite open cover of $ \con{I}$, and using (i), we deduce that $\E_p(H) \cap I$ is finite. This, in turn, yields the claimed assertion concerning accumulation points of $\E_p(H)$.
\end{proof}

If $H \in \Cc^{1,1}(A)$, then $H$ has no singular continuous spectrum on $\tilde{\mu}_A(H)$. Recall that for each $s \in \R$, $\Hh_s(A)$ denotes the Sobolev space associated to $A$.

\begin{theo}\label{MainMourre2}
Let $H \in \Cc^{1,1}(A)$, $s>1/2$ and $\K= \Hh_s(A)$.
Then  the following LAP  holds: the holomorphic map $ \C_{\pm} \ni z \mapsto (H-z)^{-1} \in \B\left(\K,\K^*\right)$ extends to a weak$^*$-continuous map on $\C_{\pm} \cup \mu_A(H)$.  It follows that the spectrum of $H$ is purely absolutely continuous on $\mu_A(H)$ and $\sic(H) \cap \tilde{\mu}_A(H)= \emptyset$.
\end{theo}

\begin{rem}
   In Theorem \ref{MainMourre2}, a LAP can be guaranteed for the Banach space $(D(A),\Hh)_{1/2,1}$, obtained via real interpolation (see \cite{abmg} for a proof). Moreover, similar observations to those made in Remark \ref{RmkLAP1} hold here as well. 
\end{rem}

Theorems \ref{MainMourre1} and \ref{MainMourre2} provide useful connections between the spectral properties of $H$ and the regularity of $H$ with respect to $A$. Now, a natural question is whether these spectral properties remain stable under suitable compact perturbations. In this context, it is useful to consider the following definition.

\begin{defi}\label{DefCritandAdmiss}
Let    $ H \in \Cc^{1,1}(A)$. We say that an operator $V \in \B(\Hh)$ is an  {admissible perturbation for $H$ with respect to} $A$ provided that $V$ is self-adjoint and $V  \in \K(\Hh) \cap \Cc^{1,1}(
A)$.
\end{defi}

The next proposition says that the spectral properties of $H$ are essentially stable under admissible perturbations. 

\begin{prop}\label{MainMourrePerturbed}
    Assume that $ H \in \Cc^{1,1}(A)$ and $V$ is an admissible perturbation for $H$ with respect to $A$, and let $\tilde{\kappa}_A(H,V):=  \kappa_A(H) \cup \E_p(H) \cup E_p(H+V)$. Then $\sigma_{\ess}(H+V) = \sigma_{\ess}(H)$  and  the following assertions hold:
    \begin{itemize}
    \item[(i)] $H+V \in \Cc^{1,1}(
A)$, $[iA, V]$ is a compact operator, $\tilde{\mu}_A(H)=\tilde{\mu}_A(H+V)$ and $ \kappa_A(H)=\kappa_A(H+V)$.
        \item[(ii)] The conclusions from Theorem \ref{MainMourre1} hold for $H$ and $H+V$. In particular, the set  $\tilde{\kappa}_A(H,V)$ is closed.
        \item[(iii)] For all $s>1/2$ and $\K= \Hh_s(A)$, the following LAP   is satisfied: the holomorphic map $ \C_{\pm} \ni z \mapsto (H-z)^{-1} \in \B\left(\K,\K^*\right)$ extends to a weak$^*$-continuous map on $\R \setminus \tilde{\kappa}_A(H,V)$.  In particular, if $\tilde{\kappa}_A(H,V)$ is countable then $\sic(H+V) = \emptyset$ . 
    \end{itemize}
\end{prop}

\begin{proof}
Thanks to the Weyl criterion we know that $\sigma_{\ess}(H+V) = \sigma_{\ess}(H)$. Assertion (i)  follows from Remark \ref{starAlgebras}, \cite[Proposition 2.1]{g-m-1} and \cite[Theorem 7.2.9]
{abmg}. The statement given in (ii) is a direct consequence of the fact that $H+V \in \Cc^{1,1}(
A) \subseteq C^1(A) $. Finally, (iii) follows from Theorem \ref{MainMourre2} and the inclusion $\sigma_{\ess}(H) \setminus \tilde{\kappa}_A(H,V) \subseteq \mu_A(H) \cap \mu_A(H+V)$. 
\end{proof}

\begin{rem} In Proposition \ref{MainMourrePerturbed}, the sets $\mu_A(H)$ and $\mu_A(H+V)$ may differ. Observe that the possible eigenvalues of $H+V$ embedded in $\sigma_{\ess}(H)$ are included in $\kappa_A(H) \cup (\tilde{\mu}_A(H) \setminus {\mu}_A(H+V))$.
\end{rem}

We close this section by defining another regularity class with respect to $A$, which provides a criterion to deal with the class $\Cc^{1,1}(A)$.

\begin{defi} We say that $T \in \B\left(\H \right)$ is of class  $\Cc^{0,1}(A)$ if
    $$\int_{0}^1 \| e^{iAt} T e^{-iAt} -T \| \, \frac{dt}{t} < \infty. $$
In this case, we write $T \in \Cc^{0,1}(A)$.
\end{defi}

\begin{rem}\label{RmkC01}
The above class is a linear subspace of $\B\left(\H\right)$, stable under adjuntion. Furthermore, according to inclusion (5.2.19) of \cite{abmg}, if $T \in C^1(A)$ and $[iA, T] \in \Cc^{0,1}(A)$, then $ T \in \Cc^{1,1}(A)$.
\end{rem}

\section{Mourre estimates for \texorpdfstring{$H_0^{(\G)}$}{ }}\label{MourreEstimatesFreeModel}

This section is concerned with the deduction of Mourre estimates on some subintervals of the essential spectrum of $H_0^{(\G)}$ for suitable conjugate operators. To lighten the presentation, throughout this section, we let $H_0:=H_0^{(\Z)}$,  $H_0^+:=H_0^{(\Z_+)}$, $X:=X^{(\Z)}$ and $X^+:=X^{(\Z_+)}$. We also denote by $\chi_\Lambda$ the characteristic function on a given set $ \Lambda \subseteq \R$.

\subsection{Conjugate operators for \texorpdfstring{$H_0$}{ }}\label{seq1}

In this section, we present the construction of a family of conjugate operators for $H_0$, based on ideas from \cite{g-m-2}. As in Section \ref{FreeModel}, let $\varphi := \varphi_2 - \varphi_1$, where  $\varphi_1,  \varphi_2 \in (-\pi, \pi]$ are such that  $a=|a|e^{i \varphi_1}$ and $  b= |b|e^{i \varphi_2}$. For each $k \in \Z_+ \setminus \{0\}$, let  $A_k$ be the symmetric  operator acting on $\ell^2_{c}(\Z; \C^2) $ by
\begin{equation}\label{SequenceAkinZ}
      -\frac{|a||b|}{4i}\left[\left(e^{ik \varphi}S^k - e^{-ik \varphi}S^{-k} \right)X + X\left(e^{ik \varphi}S^k - e^{-ik \varphi}S^{-k} \right) \right].
\end{equation}

An adaptation of the arguments given in \cite{g-g} shows that $A_k$ is essentially self-adjoint; we also denote by $A_k$ its unique self-adjoint extension. In order to formulate Mourre estimates for the operators $H_0$ and $A_k$ in  a concise way, let us introduce the functions $g_k: \R \setminus\{0\} \to \R$, for $k \geq 1$, defined for  $t \in \R \setminus \{ 0 \}$ by
\begin{equation}\label{gk}
    g_k(t)= g_0(t)U_{k-1}\left( \frac{t^2-\alpha^2-|a|^2-|b|^2}{2 |a||b|}\right),
\end{equation}  
and
\begin{align*}
    g_0(t)     
    &= -\frac{1}{4|t|} \left(t - \sqrt{\alpha^2+(|a|-|b|)^2} \right) \left(t + \sqrt{\alpha^2+(|a|-|b|)^2} \right) \times   \\
    & \quad \quad  \quad \quad  \left(t - \sqrt{\alpha^2+(|a|+|b|)^2} \right) \left(t + \sqrt{\alpha^2+(|a|+|b|)^2} \right). 
\end{align*}
Here $U_k$ stands for the $k$-th polynomial of Chebyshev of the second kind. This sequence of polynomials is defined by the recurrence relation
$$U_0(t):=1, \ U_1(t):=2t, \ U_{k+1}(t)= 2t U_{k}(t)-U_{k-1}(t), \quad k>0, \ t \in \R.$$
Furthermore, they satisfy the identity 
\begin{equation}\label{sinUk}
    \sin(t) U_{k-1}(\cos(t)) = \sin(k t), \quad k>0, \ t \in \R.
\end{equation}

With this notation, we formulate the first result of this section.

\begin{theo}\label{MainAk}
Let $I_{\pm}$ be the spectral bands of the operator $H_0$ defined in \eqref{leftband},  $k \in \Z_{+} \setminus\{0\}$ and $g_k$ be defined as in \eqref{gk}. The operator $H_{0}$ is of class $C^\infty(A_k)$. Furthermore, for any Borel set    $\Lambda_\pm \subseteq \R$ such that  $\con{\Lambda_\pm} \subseteq I_\pm$ it holds 
\begin{equation}\label{MEAkwithGap1}
E_{\Lambda_{\pm}}(H_0 ) [ i{A}_k, H_0]  E_{\Lambda_{\pm}}(H_0 )= \pm g_k(H_0)  E_{\Lambda_{\pm}}(H_0 ).
\end{equation}
In particular, if for each $j \in \{0, 1, \ldots, k\}$ we let  $\theta_j \in [0, \pi]$ be such that $  \cos (\theta_j + \varphi)= \cos\left( \frac{\pi j}{k} \right),$ then
\begin{equation}\label{kthetaj}
    \kappa_{ A_k}(H_0)
 =  \left\{ \pm \lambda\left( \theta_j \right) : j=0,1,\ldots,k\right\},
\end{equation}
and
\begin{equation}\label{muthetaj}
    \mu^+_{A_k}(H_0) = \bigcup_{j=0}^{\floor*{\frac{k-1}{2}}} \left( \lambda(\theta_{2j+1}),  \lambda(\theta_{2j}) \right), \quad  \quad \mu^{-}_{A_k}(H_0) = \bigcup_{j=0}^{\floor*{\frac{k}{2}}} \left( \lambda(\theta_{2j}),  \lambda(\theta_{2j-1}) \right),
\end{equation}
Moreover, $ \tilde{\mu}_{A_k}(H_0) = \mu_{A_k}(H_0)$.
\end{theo}

\begin{proof} We first show that  ${H}_0 \in C^\infty({A}_k)$. Let $ \F : \H_\Z \to L^2(\T;\C^2)$ be the Fourier transform defined in \ref{FourierT},  $\widehat{H}_0:=  \F H_0 \F^{-1}$ and $ \widehat{A}_k := \F A_k \F^{-1}$. Note that  $ \widehat{A}_k$  agrees with the unique self-adjoint extension to $L^2(\T;\C^2)$ of the symmetric operator 
\begin{equation*}\label{hatAk}
  -\frac{|a| |b| }{2}\left[  \sin(k(\theta + \varphi))(-i \partial_\theta) + (-i \partial_\theta) \sin(k(\theta + \varphi)) \right],
\end{equation*}
defined on $C^\infty(\T; \C^2)$. Here $ -i\partial_\theta$ stands for the first derivative operator on $L^2(\T;\C^2)$. Since $\widehat{H}_0$ is the multiplication operator by the smooth matrix-valued symbol 
$h$ defined in \eqref{symbol}, we know that $\widehat{H}_0 \in C^\infty \left(\widehat{A}_k \right)$. From this, we conclude that ${H}_0 \in C^\infty({A}_k)$.

Regarding identity \eqref{MEAkwithGap1}, we only consider the case  $ \con{\Lambda_+} \subseteq I_+$. The identity for $ \con{\Lambda_-} \subseteq I_-$ can be handled analogously. Let $\Lambda \subseteq \R$ be a Borel set such that $\con{\Lambda} \subseteq I_+$. Consider $ {\eta} \in C_c^\infty(  {\R_{+}})$ such that $ {\eta} \chi_{\Lambda} = \chi_{\Lambda}$ and $ \eta' \chi_{\Lambda} = 0$.  So, 
\begin{align*}
  E_{\Lambda}\left({\widehat{H}_0}\right)  \left[ i\widehat{A}_k, \widehat{H}_0 \right]   E_{\Lambda}\left({\widehat{H}_0}\right)   &=  E_{\Lambda}\left({\widehat{H}_0}\right)   \left[i \widehat{A}_k,  {\eta} ( \widehat{H}_0) \widehat{H}_0 \right] E_{\Lambda}\left({\widehat{H}_0}\right) . 
\end{align*}
For each $\theta \in \T$, let  $\Pi(\theta)$ and $\Pi^\perp (\theta)$ be the orthogonal projections on the eigenspace of the symbol $h(\theta)$ associated to the eigenvalues $ \lambda (\theta) $ and $- \lambda (\theta) $, respectively. We define the set $N:= \{ \theta \in \T : \lambda(\theta) = 0 \}$, which contains at most one point. Note that  
\begin{align*} 
\Pi (\theta) &= \frac{1}{2} \left( \frac{\widehat{H}_0(\theta)}{\lambda (\theta) } + 1  \right), \quad \theta \in \T \setminus N.
\end{align*}
Hence, for all $f \in L^2(\T; \C^2)$ and $\theta \in \T$ we have 
\begin{align*}
\left(\widehat{H}_0 f \right)(\theta)  &= \chi_{\T \setminus N}(\theta)  \left( \lambda(\theta) \Pi(\theta)   -\lambda(\theta)  \Pi ^\perp(\theta)  \right)f(\theta).
\end{align*}
 
Furthermore, the following operators can also be written as multiplication operators on $L^2(\T; \C^2)$ by matrix-valued functions: $|\widehat{H}_0|= \lambda,$
$$
\begin{aligned}
     E_\Lambda \left({\widehat{H}_0}\right)  &= (\chi_\Lambda \circ \lambda) \Pi, &&
{\eta}\left({\widehat{H}_0}\right)   = ( {\eta} \circ \lambda) \Pi,  \\
  {\eta} \left({\widehat{H}_0}\right) \widehat{H}_0  &= ( {\eta} \circ \lambda) \lambda \Pi, && 
  {\eta'} \left({\widehat{H}_0}\right)   = ( {\eta'} \circ \lambda) \Pi.
\end{aligned}
$$
 Thus, 
\begin{align}
    \left[i \widehat{A}_k, {\eta} \left({\widehat{H}_0}\right) \widehat{H}_0\right] &=- |a| |b| \sin(k(\theta + \varphi))\left[ i(-i\partial_\theta), ({\eta}   \circ \lambda) \lambda \Pi \right] \notag \\
     &= - |a| |b| \sin(k(\theta + \varphi))\left[({\eta} ' \circ \lambda) \lambda'  \lambda \Pi + ({\eta}  \circ \lambda) \lambda' \Pi + ({\eta}   \circ \lambda) \lambda \Pi'\right] \notag \\
     & = - |a| |b| \sin(k(\theta + \varphi))\left[ \lambda'  \lambda {\eta} '\left({\widehat{H}_0}\right)  + \lambda' {\eta} \left({\widehat{H}_0}\right)  +  \lambda({\eta}   \circ \lambda) \Pi'\right]. \label{proofe1}
\end{align}
By the choice of $\eta$, we know that $ E_{\Lambda}\left({\widehat{H}_0}\right) \eta' \left({\widehat{H}_0}\right) \Pi'  E_{\Lambda}\left({\widehat{H}_0}\right)=0. $
Since  $\Pi(\theta) = \Pi^2(\theta)$ for $\theta \in \T \setminus N $, we have $(\Pi \Pi' \Pi f)(\theta) =0$ for $f \in L^2(\T; \C^2)$  and $\theta \in \T \setminus N $, and this implies
\begin{align*}
   E_{\Lambda}\left({\widehat{H}_0}\right)  \Pi'  E_{\Lambda}\left({\widehat{H}_0}\right)    =   E_{\Lambda}\left({\widehat{H}_0}\right)  \Pi  \Pi' \Pi  E_{\Lambda}\left({\widehat{H}_0}\right)  = 0.
\end{align*}
 
Hence from \eqref{proofe1} together with identities \eqref{EigvSymbol} and \eqref{sinUk},  we get
\begin{align*}
    E_{\Lambda}\left({\widehat{H}_0}\right) \left[i \widehat{A}_k,  \widehat{H}_0 \right]  E_{\Lambda}\left({\widehat{H}_0}\right) & =U_{k-1}(\cos(\theta + \varphi)) \frac{|a|^2|b|^2\sin^2(\theta + \varphi)}{\lambda(\theta)}  E_{\Lambda}\left({\widehat{H}_0}\right)  \\
    & = g_k\left({\widehat{H}_0}\right) E_{\Lambda}\left({\widehat{H}_0}\right) .
\end{align*}
By taking the Fourier transform, we obtain 
$
      E_{\Lambda}({H}_0)  \left[i  {A}_k,  {H}_0 \right]    E_{\Lambda}({{H}_0}) = g_k( {H}_0) E_{\Lambda}({ {H}_0}).
$

Now, since $\E_p(H_0) = \emptyset$ (see Proposition \ref{Prop1}), by Theorem \ref{MainMourre1}  we have that $ \mu_{A_k}(H_0)=  \tilde{\mu}_{A_k}(H_0)$. So, by \eqref{MEAkwithGap1} we conclude that  $\kappa_{ A_k}(H_0) = \sigma(H_0) \setminus \mu_A(H_0) =\{ t \in \sigma(H_0) : g(t)=0\}$. Therefore,  \eqref{kthetaj} and \eqref{muthetaj} follows from the fact that  the roots of the function $\R \ni x \mapsto (1-x^2)U_{k-1}(x)$ are given by 
$$x_j:=\cos\left( \frac{\pi j}{k} \right), \quad j=0,1, \ldots,k,$$
and it is positive on  
$ \bigcup_{j=0}^{\floor*{\frac{k-1}{2}}} \left( x_{2j+1}, x_{2j} \right)$ and negative on $\bigcup_{j=0}^{\floor*{\frac{k}{2}}} \left(x_{2j},  x_{2j-1} \right).$ 
\end{proof}

\begin{rem} If $H_0$ exhibits a spectral gap, then identity \eqref{MEAkwithGap1} is true for $\Lambda_\pm:=I_{\pm}$. 
 
\end{rem}

\subsection{Conjugate operators for \texorpdfstring{$H_0^+$}{ } }\label{seq2}

The aim now is to construct a family of conjugate operators for $H_0^+$. With a slight abuse of notation, we denote the canonical orthonormal bases of $\ell^2(\Z_+)$ and $\ell^2(\Z)$ by $(\delta_n)_{n \in \Z^+}$ and $(\delta_n)_{n \in \Z}$, respectively. By Remark \ref{rmkH1},  a natural orthonormal basis of  $\H_\G$ is given by $\left\{ \delta_n^-, \delta_n^+: n \in \G  \right\}$, where
\begin{align*}
\delta_{n}^- = \left( \begin{array}{c} 
\delta_{n} \\
0 \end{array} \right)
 \quad \text{and} \quad 
\delta_{n}^+ = \left( \begin{array}{c} 
0 \\
\delta_{n} \end{array} \right), \quad n\in \G.
\end{align*}
Moreover, if $P \in \B\left(\H_{\Z} \right)$ is the orthogonal projection of $\H_{\Z}$ onto $\H_{\Z_+}$, then
$$P= \sum_{j=0}^\infty \left( |  \delta_j^{-}  \ket   \bra \delta_j^{-}  | +  | \delta_j^{+}   \ket  \bra \delta_j^{+}  |\right).$$
Mind that a linear operator $T$ acting on a subspace of $\H_{\Z_+}$ can be canonically identified with  $ {P TP}$. Conversely, if $T$ is a linear operator acting on a subspace of $ \H_\Z$, then ${P TP}$ can be considered as a linear operator on  $\H_{\Z_+}$.

Now, for each $k \in \Z_+ \setminus \{0\} $, let $A_k^+$ be the self-adjoint operator on $\H_{\Z_+}$ defined by 
\begin{equation*}
    A_k^+:= P A_k P, \quad k>0,
\end{equation*}
where $(A_k)_{k >0}$ is the sequence of operators on $\H_{\Z}$ defined in Section \ref{seq1}. The operator $ A_k^+$ agrees with the closure of the essentially  self-adjoint operator defined on $\ell^2_{c}(\Z_+; \C^2)$ by
\begin{equation}\label{SequenceAkinZplus}
     -\frac{|a||b|}{4i}\left[ \left(e^{ik \varphi}S^k - e^{-ik \varphi}S^{*k} \right)X^+ + X^+\left(e^{ik \varphi}S^k - e^{-ik \varphi}S^{*k} \right) \right].
\end{equation}

We require two auxiliary lemmas to prove that $\left(A_k^+\right)_{k>0}$ is indeed a sequence of conjugate operators for $H_0^+$. The scalar version of the next result was proved in \cite[Lemma 3.7]{o-toe}.

\begin{lem}\label{FuncH0} The operators  $P H_0 P^\perp$ and $P^\perp H_0 P$ are of finite rank. Moreover, if $\Phi: \R \to \C$ is a continuous function of compact support, then
    \begin{itemize}
        \item[(i)] $ \Phi(H_0^{+}) \simeq P \Phi(H_0)P$.
        \item[(ii)] $P \Phi(H_0)P^\perp$ and $ P^\perp\Phi(H_0)P$ are compact. 
    \end{itemize}
\end{lem}

\begin{proof}
    Simple calculations show that 
    $$P H_0 P^\perp =  \con{b} \left| \delta_0^{-} \right\ket \left\bra \delta_{-1}^{-} \right| \quad \mbox{and} \quad    P^\perp H_0 P = b \left| \delta_{-1}^+ \right\ket \left\bra \delta_0^+ \right|    .$$
According to the Stone-Weierstrass Theorem, to show statements (i) and (ii), it is enough to consider the case when $\Phi$ is a polynomial. We prove by induction that $\left(H_0^+ \right)^j \simeq P (H_0)^j P$ for $j \in \Z_{+}$. This is true for $j=0$. Assume the induction hypothesis for some $j \geq 0$. Then 
    $$\left(H_0^+ \right)^{j+1} = \left(H_0^+ \right) ^j(PH_0P) \simeq  P (H_0)^j P  H_0 P =  P (H_0)^{j} (1 - P^\perp )(H_0)P \simeq P (H_0)^{j+1} P.$$
    Thus statement (i) is true if $\Phi$ is a polynomial. 
    
We also proceed by induction to verify that $P(H_0)^jP^\perp \simeq 0$ for $j \in \Z_{+}$. This is clear for $j=0$. Suppose the induction hypothesis is true for some $j \geq 0$. Then
$$ P(H_0)^{j+1}P^\perp = P(H_0)^jP P H_0P^\perp  + P(H_0)^jP^\perp P^\perp H_0P^\perp \simeq 0. $$
Hence $P \Phi(H_0) P^\perp \simeq 0$ if $\Phi$ is a polynomial. Finally, note that $P^\perp \Phi(H_0) P = (P \con{\Phi}(H_0) P^\perp)^*  $ is also compact.
\end{proof}

\begin{lem}\label{PAkPperp}
Let $k \in \Z_{+}$, with $k>0$ . The operators $PA_kP^\perp$ and $ P^\perp A_kP $ are of finite-rank, and their ranges are included in $\D\left(A_k \right)$.
\end{lem}

\begin{proof}
Direct calculations show that
$$ P A_k P^\perp = - \frac{|a||b|}{4i} \sum_{j=0}^{k-1} (2j-k) e^{ik \varphi}\left( \left|\delta_{j}^- \right\ket \left\bra \delta_{j-k}^-\right| + \left|\delta_{j}^+ \right\ket \left\bra \delta_{j-k}^+ \right|   \right)   $$
and
$$ P^\perp A_k P=  \frac{|a||b|}{4i} \sum_{j=0} ^{k-1} (2j-k) e^{-ik \varphi}\left( \left|\delta_{j-k}^- \right\ket \left\bra \delta_j^- \right| + \left|\delta_{j-k}^+ \right\ket \left\bra \delta_j^+ \right|  \right). $$
The conclusion follows.
\end{proof}

\begin{theo}\label{MourreH0+}
    Let $k \in \Z_{+}\setminus \{0\}$. Then $H_0^+$ is of class $C^2\left(A_k^+\right)$ and  $ \left[iA_k^+, H_0^+ \right] \simeq  P \left[iA_k, H_0 \right] P $. 
\end{theo}
    
\begin{proof}
    By Theorem \ref{MainAk}, $H_0 \in C^\infty(A_k)$. So
    $$\left[iA_k^+, H_0^+ \right]  = [i P A_k P, PH_0P] = P [iA_k, H_0] P - i P A_k P^\perp H_0 P + iP H_0 P^\perp  A_kP.$$
By Lemma \ref{PAkPperp}, this means that $[iA_k^+, H_0^+] \simeq P [iA_k, H_0] P $. For the second order commutator, we have that 
   \begin{align*}
       \left[iA_k^+, \left[iA_k^+, H_0^+\right]\right] -  P\left[iA_k, [iA_k, H_0]\right] P & = -i PA_k P^\perp [iA_k, H_0] P +i P [iA_k, H_0] P^\perp A_k P   \\
       & \quad + PA_kPA_kP^\perp H_0P  + PH_0P^\perp A_k P A_k P\\
       & \quad - PA_kP^\perp H_0PA_kP   - P A_k P H_0 P^\perp A_k P.
   \end{align*}
  By virtue of Lemmas \ref{FuncH0} and \ref{PAkPperp}, the operators on the right side are of finite rank. In particular, $ \left[iA_k^+, \left[iA_k^+, H_0^+\right] \right] \simeq  P[iA_k, [iA_k, H_0]] P $. This means that $H_0^+$ is of class $C^2\left(A_k^+ \right)$. This concludes the proof. 
\end{proof}

The next theorem is the analog of Theorem \ref{MainAk} for $H_0^+$ and $A_k^+$ instead of $H_0$ and $A_k$, respectively.

\begin{theo}\label{MainAkUnilateral}
Let $I_{\pm}$ be the spectral bands of the operator $H_0^+$ defined in \eqref{leftband},  $k \in \Z_{+} \setminus\{0\}$ and $g_k$ be defined as in \eqref{gk}. For any Borel set $\Lambda_\pm \subseteq \R$ such that   $\con{\Lambda_\pm} \subseteq I_\pm$ it holds
\begin{equation}\label{MEAkUnilateral}
E_{\Lambda_{\pm}} \left(H_0^+ \right) \left[ i{A}^+_k, H_0^+ \right]  E_{\Lambda_{\pm}}\left(H_0^+ \right) \simeq \pm g_k(H_0^+)  E_{\Lambda_{\pm}}\left(H_0^+ \right).
\end{equation}
In particular,  $\kappa_{ A_k^+}\left(H_0^+ \right) =  \kappa_{ A_k}(H_0) $ and $  \tilde{\mu}_{A_k^+}\left(H_0^+\right) =\mu_{A_k^+}\left(H_0^+\right) =  \mu_{A_k}(H_0)$.  
\end{theo}

\begin{proof} We show \eqref{MEAkUnilateral} for $\Lambda_+$. The proof for $\Lambda_{-}$ follows the same argument. Let $\Delta$ be an open interval such that $\con{\Delta} \subseteq  I_+$ and  $\con{\Lambda_{+}} \subseteq \Delta \subseteq I_+$. By Urysohn Lemma, there exists a continuous function $\Phi: \R \to \C$ that takes the value $1$ on $\con{\Lambda_+}$ and vanishes on $\R \setminus \Delta$. Thus $\Phi(H_0)E_{\Delta}(H_0) = \Phi(H_0) $ and from  \eqref{MEAkwithGap1}, applied to $\Delta$, we  get
 \begin{equation*}
       \Phi(H_0) [iA_k, H_0] \Phi(H_0) =  g_k(H_0) \Phi^2(H_0).
    \end{equation*}
By Theorem \ref{MourreH0+} and Lemma \ref{FuncH0}, we have 
    \begin{align*}
       \Phi\left(H_0^+ \right) \left[iA_k^+, H_0^+\right] \Phi(H_0^+) &\simeq \Phi\left(H_0^+ \right) P [iA_k, H_0] P  \Phi \left(H_0^+ \right) \\
        &\simeq  P \Phi(H_0)  P [iA_k, H_0] P  \Phi(H_0) P  \\
        & =   P \Phi(H_0)  [iA_k, H_0]   \Phi(H_0) P - P \Phi(H_0)  P^\perp [iA_k, H_0]   \Phi(H_0) P \\
        & \quad -  P \Phi(H_0)   [iA_k, H_0] P^\perp  \Phi(H_0) P \\ & \quad  +  P \Phi(H_0)  P^\perp [iA_k, H_0] P^\perp  \Phi(H_0) P \\
        & \simeq P \Phi(H_0)  [iA_k, H_0]   \Phi(H_0) P \\
        & =  P g_k(H_0) \Phi^2(H_0)  P \\
        & \simeq  g_k\left(H_0^+\right) \Phi^2\left(H_0^+\right). 
    \end{align*}
Together  with $E_{\Lambda_+}\left(H_0^+ \right)\Phi\left(H_0^+ \right) = E_{\Lambda_+}\left(H_0^+ \right) $, this yields the claimed relation. Now, from   \eqref{MEAkwithGap1} and \eqref{MEAkUnilateral} we deduce that $\tilde{\mu}_{A_k^+}\left( H_0^+ \right)  = \mu_{A_k}(H_0)$ and $\kappa_{A_k^+}\left( H_0^+ \right) = \kappa_{A_k }\left( H_0  \right)$. Finally,  $ {\mu}_{A_k^+}\left(H_0^+\right) =\mu_{A_k}(H_0)$ follows from  Proposition \ref{Prop1} and Theorem \ref{MainMourre1}. 
    \end{proof}

\subsection{An alternative conjugate operator in the gapless case}\label{A0Gapless}

In this section, we assume that $H_0^{(\G)}$ does not exhibit a spectral gap (see Proposition \ref{Prop1}). Following Remark \ref{remDiscreteLaplacian}, a conjugate operator for $H_0^{(\Z)}$ can be derived from a conjugate operator for the discrete Laplacian via a suitable unitary equivalence. We instead adopt an intrinsic construction that is formulated directly in terms of the structure of $H_0^{(\Z)}$. 

As before, let $a = |a|e^{i \varphi_1}$, $b=|b|e^{i \varphi_2}$ and $ \varphi:= \varphi_2-\varphi_1$. We define $A_0$ as the unique self-adjoint extension on $\H_{\Z}$ of the operator
\begin{equation}\label{A0inZ}
   -|a|\begin{pmatrix}
        0 & i e^{i\varphi_1} (I - e^{i \varphi} S) \\
        -i e^{-i\varphi_1} (I - e^{-i\varphi}S^{-1})  & 0
    \end{pmatrix}X + h.c.
\end{equation}
defined initially on $\ell^2_c(\Z; \C^2)$. Here, $h.c.$ denotes the adjoint of the preceding operator.

\begin{theo}\label{MainA0} The operator $H_0$ is of class $ C^\infty(A_0)$ and $$[A_0, H_0] = 4|a|^2-H_0^2.$$
Furthermore, $\kappa_{A_0} (H_0)= \{\pm 2|a|\}$ and  $\tilde{\mu}_{A_0}(H_0) =\mu_{A_0}(H_0) = \mu^+_{A_0}(H_0) = (-2|a|,2|a|)$.
\end{theo}

\begin{proof}
 Let $ \F : \H_\Z \to L^2(\T;\C^2)$ be the Fourier transform defined in \eqref{FourierT},  $\widehat{H}_0:=  \F H_0 \F^{-1}$ and $\widehat{A}_0 := \F A_0 \F^{-1}$. Then $\widehat{H}_0$ is the multiplication operator defined by \eqref{MultOperator}, and $\widehat{A}_0$ is the unique self-adjoint extension of the operator acting on $C^\infty(\T;\C^2)$ by
\begin{equation*}\label{AkNoGap}
  M(- i \partial_\theta) +  (- i \partial_\theta)M, 
\end{equation*}
where $M$ is the multiplication operator by the matrix-valued function
\begin{equation*}
    M(\theta) =-|a| \begin{pmatrix}
        0 & i e^{i \varphi_1} (1-e ^{i(\varphi+ \theta)}) \\
        - i e^{-i \varphi_1} (1-e ^{-i(\varphi+ \theta)}) & 0
    \end{pmatrix}.
\end{equation*}
One readily verifies that $M$ commutes with $\widehat{H}_0$. Hence 
\begin{align*}
    \left[i\widehat{A}_0, \widehat{H}_0 \right] =  M \left[\partial _ \theta, \widehat{H}_0 \right] +  \left[\partial _ \theta, \widehat{H}_0 \right] M =M\widehat{H}'_0 + \widehat{H}'_0  M =2\Rp\left(M \widehat{H}'_0\right).
\end{align*}
Since for all $\theta \in \T$ it holds that
\begin{equation*}
    M(\theta) \widehat{H}'_0(\theta) = -|a|^2 \begin{pmatrix}
        e^{-i(\varphi+\theta)} -1 & 0\\
        0 & e^{i(\varphi+\theta)}-1
    \end{pmatrix}, 
\end{equation*}
we have
\begin{equation*}
2    \Rp\left(M \widehat{H}'_0\right) = 2|a|^2 (1-\cos(\varphi + \theta)).
\end{equation*}
 
Combined  with the identity $\widehat{H}_0^2= 2|a|^2 (1+\cos(\varphi + \theta))$, this gives
$$ \left[i \widehat{A}_0, \widehat{H}_0 \right] = -\widehat{H}_0^2  +2|a|^2 +2|a|^2 = 4|a|^2-\widehat{H}_0^2.$$ 
Applying the Fourier transform we get the claimed expression for  $[iA_0, H_0]$. By induction, we infer that $H_0 \in C^\infty(A_0)$. The identity $ \mu_{A_0}(H_0) = \mu^+_{A_0}(H_0) = (-2|a|,2|a|) $ is deduced from the fact that the roots of the function $\sigma(H_0) \ni t \mapsto 4|a|^2-t^2 \in \R$ are $\pm 2|a|$, and this function can be bounded from below by a positive constant on any compact interval $I \subseteq (-2|a|, 2|a|)$.  Finally, since $\E_p(H_0) = \emptyset$ (see Proposition \ref{Prop1}), by Theorem \ref{MainMourre1}  we conclude that  $ \mu_{A_0}(H_0)=  \tilde{\mu}_{A_0}(H_0)$ and $\kappa_A(H_0)=\{\pm 2|a| \}$.  
\end{proof}

Now, as in Remark \ref{rmkH1}, let $P \in \B\left(\H_{\Z}\right)$ denote the orthogonal projection of $\H_{\Z}$ onto $\H_{\Z_+}$. We define  the self-adjoint operator $A_0^+$  on $\H_{\Z_+}$ by
\begin{equation*}
    A_0^+ := PA_0P.
\end{equation*}
Then $A_0^+$ is the closure of the operator acting on $ \ell^2_c(\Z_{+}; \C^2)$ by
\begin{equation}\label{A0inZplus}
    -|a|\begin{pmatrix}
        0 & i e^{i\varphi_1} (I - e^{i \varphi} S) \\
        -i e^{-i\varphi_1} (I - e^{-i\varphi}S^*)  
    \end{pmatrix}X^+ + h.c.
\end{equation}
The next result is the analog of Theorem \ref{MainA0} for $H_0^+$ and $A_0^+$. Observe in particular that by Proposition \ref{Prop1} and Theorem \ref{MainMourre1}, $\tilde{\mu}_{A_0^+}(H_0^+) = \mu_{A_0^+}(H_0^+)$ because $ \E_p \left( H_0^+\right)=\emptyset$.

\begin{theo}\label{MainA0Unilateral} The operator $H_0^+$ is of class $ C^\infty(A_0^+)$ and $$\left[A_0^+, H_0^+ \right] \simeq 4|a|^2-(H_0^+)^2.$$
Furthermore,  $\tilde{\mu}_{A_0^+}(H_0^+) = \mu_{A_0^+}(H_0^+) = \mu^+_{A_0^+}(H_0^+) = (-2|a|,2|a|)$ and $\kappa(H_0^+, A_0^+)= \{\pm 2|a|\}$.
\end{theo}

\begin{proof}
The proof follows from Theorem \ref{MainA0} and arguments analogous to those used in the proofs of Theorems~\ref{MourreH0+} and \ref{MainAkUnilateral}.     
\end{proof}

\section{Admissible perturbations}\label{admissiblesboth}
In this section, we describe the admissible perturbations for  $H_0^{(\G)}$ with respect to the conjugate operators defined in Section \ref{MourreEstimatesFreeModel}. To unify the notation, we define $A_k^{(\Z)}:= A_k$ and $A_k^{(\Z_{+})}:= A_k^+$ for each $k \in \Z_+$. Recall that if $k>0$, then  $A_k^{(\G)}$  is the unique self-adjoint extension of the operator defined in $\ell^2_c(\G; \C^2)$  by \eqref{SequenceAkinZ}, when $\G=\Z$, and \eqref{SequenceAkinZplus} when $\G=\Z_+$. While, the operator $A_0^{(\G)}$ is the self-adjoint extension of the operator acting in $\ell^2_c(\G; \C^2)$ by \eqref{A0inZ} or \eqref{A0inZplus}, depending on whether $\G = \Z$  or $\G=\Z_+$. The next lemma concerns the regularity of the shift operators on $\H_\G$ with respect to each $A_k^{(\G)}$.

\begin{lem}\label{shiftRegularity}
    The shift operators $S$ and $S^{*}$ on $\H_{\G}$ are of class $C^{\infty}\left(A_k^{(\G)} \right)$ for each $k \in \Z_+$.   
\end{lem}

\begin{proof} Let $k \in \Z_+$ and define $F^{(\Z)} := 0 $ and $ F^{(\Z_+)} := \chi_{\{0 \}}\left(X^{(\Z_+)} \right)  I_2 $. On the subspace $\ell^2_c(\G; \C^2)$ we have
\begin{align*}
      \frac{4}{|a||b|}  \left[iA_k^{(\G)}, S\right] &= -2  e^{ik \varphi} S^{k+1}  + 2e^{-ik \varphi} S^{*(k-1)}   + e^{-ik \varphi} (k-2) F^{(\G)}   S^{*(k-1)} 
   \end{align*}
for $k>0$, and
\begin{align*}
     -\frac{1}{|a|} \left[iA_0^{(\G)}, S\right] & = 2 S \begin{pmatrix}
         0 & -e^{i \varphi_1} \\
         e^{-i \varphi_1} & 0
     \end{pmatrix} + 2 S^{2} \begin{pmatrix}
         0 & e^{i \varphi_2} \\
         0 & 0
     \end{pmatrix} -  \left(2- F^{(\G)}  \right)\begin{pmatrix}
         0 & 0 \\
         e^{-i \varphi_2} & 0
     \end{pmatrix}.
\end{align*}
By density, these identities remain true on $\H_\G$. Thus, $S \in C^{1}\left(A_k^{(\G)} \right)$, and hence $S^{*} \in C^{1}\left(A_k^{(\G)} \right)$. Note that $$ \frac{4}{|a||b|}  \left[iA_k^{(\G)}, F^{(\Z_+)}\right] = ke^{ik\varphi} S^k  F^{(\Z_+)}  -ke^{-ik\varphi}  F^{(\Z_+)} S^{*k}$$
for $k>0$, and
$$   -\frac{1}{|a|} \left[iA_0^{(\G)}, F^{(\Z_+)} \right]  =  S F^{(\Z_+)} \begin{pmatrix}
    0 & e^{i\varphi_2} \\
    0 & 0
\end{pmatrix} +  \begin{pmatrix}
    0 & 0 \\
     e^{-i\varphi_2}  & 0
\end{pmatrix} F^{(\Z_+)} S^*.  $$
An argument by induction yields that $S$ and $S^{*}$ are of class $C^{\infty}\left(A_k^{(\G)} \right).$ 
\end{proof}

From Remark \ref{starAlgebras} and Lemma \ref{shiftRegularity}, we know that if   $V \in \K(\H_{\G})$ is given by \eqref{GeneralPotential}, then $V \in \Cc^{1,1}\left(A_k^{(\G)}\right) $ or $V \in C^m\left(A_k^{(\G)}\right)$,  for some $m \in \Z_+$,  provided that for each $j=0,1, \ldots,N$, the component $V_j$ belongs to the same class. The next lemma gives criteria for the classes $C^1\left(A_k^{(\G)}\right)$ and $C^2\left(A_k^{(\G)}\right)$. In this result, we consider the sets $\Qq_{0,1}(\G)$ and $\Qq_{k,m}(\G)$, with $k>0$ and $m \in \{1,2\}$,   defined in \eqref{Q0set} and \eqref{Qkset}, respectively.

\begin{lem}\label{ClasC1Gap} Let $V \in \K(\H_{\G})$ be as in \eqref{GeneralPotential}, $k \in \Z_+$ and $m \in \{1,2\}$.

\begin{itemize}
    \item[(i)] If $k>0$ and   $V_j \in \Qq_{k,m}(\G)$ for each $j=0,1, \ldots,N$, then $V \in C^m\left(A_k^{(\G)}\right)$.
    \item[(ii)] If $k=0$ and  $V_j \in \Qq_{0,1}(\G)$  for all $j=0,1, \ldots, N$, then $V \in C^1\left(A_0^{(\G)}\right)$. 
\end{itemize}

\end{lem}

\begin{proof} Let $ W \in \ell^\infty\left(\G; M_2(\C)\right)$. We first show (i). Suppose that $k>0$ and $W \in \Qq_{k,1}(\G)$. On the space $\ell^2_c(\G; \C^2)$ one has
\begin{equation}\label{FirstCommutatorAk}
    \frac{4}{|a||b|} \left [i  A_k^{(\G)}, W \right] =    e^{i k \varphi} S^k
    (2X^{(\G)}+k)(\tau^{-k} W-W)+(2X^{(\G)}+k) (\tau^{-k} W -W) e^{-ik \varphi}S^{*k} 
\end{equation}
From this, we deduce that  $W \in C^1\left(A_k^{(\G)}\right)$.

Now suppose that $W \in \Qq_{k,2}(\G)$. On the space $\ell^2(\G; \C^2)$ we have that 
\begin{align*}
    \left[[i{A}_k^{(\G)}, \left[i{A}_k^{(\G)},W\right]\right] & = -\frac{1}{16} |a|^2|b|^2 e^{i\varphi X^{(\G)}}  S^{2k}(\mathfrak{D}_1(W)+ \mathfrak{D}_{-1}(W))  e^{-i\varphi X^{(\G)}} -\frac{1}{8} |a|^2|b|^2  \mathfrak{D}_0(W) \\
      & \quad -\frac{1}{16} |a|^2|b|^2  (\mathfrak{D}_1(W)+ \mathfrak{D}_{-1}(W)) S^{*2k} e^{-i\varphi X^{(\G)}}
\end{align*}
where $\mathfrak{D}_j(W)$ are multiplication operators on $\H_{\G}$ given by
\begin{align*}
     \mathfrak{D}_1(W) & := X^{(\G)}(2X^{(\G)}+k)(2 \tau^{-k}W -W - \tau^{-2k}W) + k(2X^{(\G)}+k) (\tau^{-k}W -W  )  \\
  &  \quad  -2kX^{(\G)} (\tau^{-2k}W -\tau^{-k}W)  \\
  \mathfrak{D}_0(W) & := (4 (X^{(\G)})^2+k^2) (2W-\tau^k W -\tau^{-k}W) - 4X^{(\G)} k(\tau^{-k}W -\tau^{k}W ) \\
&   \quad -(2x-K)^2 F_k^{(\G)}, \\
  \mathfrak{D}_{-1}(W) & :=  X^{(\G)}(2X^{(\G)}+k)(2 \tau^{-k}W -W - \tau^{-2k}W)  +2k(2X^{(\G)} +k)(\tau^{-k}W-W) \\
  & \quad -k(4  X^{(\G)}+3k) (\tau^{-2k}W -\tau^{-k}W).
\end{align*}
Here we let $F_k^{(\Z)} :=0$ and $F_k^{(\Z_+)}:
 = \chi_{(-\infty, k)}( X^{\Z_+})$ for each  $k>0$. Hence $ W \in C^2\left(A_k^{(\G)}\right)$.
 
Therefore, by Lemma \ref{shiftRegularity}, we infer that if $m \in \{1,2\}$ and $V_j \in \Qq_{k,m}(\G)$ for each $j=0,1, \ldots,N$, then $V \in C^m\left(A_k^{(\G)}\right)$. This finishes the proof of (i).

Finally, we show (ii). Assume that $W \in \Qq_{0,1}(\G)$.  On the space $\ell_c^2(\G;\C^2)$ one has that
    \begin{equation}\label{CommutatorA0}
        \left[iA_0^{(\G)}, W \right] = -|a| (e^{i \varphi_2}S  (2 X^{(\G)}+1) B_1(W) -2   X^{(\G)} B_0(W)+ e^{-i \varphi_2}(2 X^{(\G)}+1)B_{-1}(W)S^{*}),
    \end{equation}
where $B_1(W)$, $B_0(W)$ and $B_{-1}(W)$ are multiplication operators on $\ell_c^2(\G;\C^2)$ given  by
\begin{align} \label{Dzero1}
    B_1(W) & :=  \begin{pmatrix}
        W^{21} & W^{22}- \tau^{-1}W^{11} \\
        0 & -  \tau^{-1} W^{21} 
    \end{pmatrix} \\ \label{Dzero2}
    B_0(W) & :=\begin{pmatrix} 
        (e^{i \varphi_1}W^{21}+e^{-i \varphi_1}W^{12}) & e^{i \varphi_1}(W^{22}-W^{11}) \\
        e^{-i \varphi_1}(W^{22}-W^{11}) & -(e^{i \varphi_1}W^{21}+e^{-i \varphi_1}W^{12})
    \end{pmatrix} \\ \label{Dzero3}
    B_{-1}(W) & := \begin{pmatrix}
        W^{12} & 0\\
        W^{22}- \tau^{-1}W^{11}  & - \tau^{-1} W^{12}
    \end{pmatrix}.
\end{align}
By density, we conclude that $W \in C^1\left(A_k^{(\G)}\right)$. Together with Lemma~\ref{shiftRegularity}, this completes the proof.
\end{proof}

The following extension of \cite[Lemma~5.1]{mandich2017} is used to show a criterion for a perturbation $V$ belonging to  $\mathcal{C}^{1,1}\left(A_k^{(\G)}\right)$.

\begin{lem}\label{AtoX} 
    For all $k \in \Z_{+}$ and $s>0$, $\left\bra A_k^{(\G)} \right\ket^s \left\bra X^{(\G)} \right\ket^{-s} $ and $\left\bra X^{(\G)} \right\ket^{-s} \left\bra A_k^{(\G)} \right\ket^s$ are bounded operators in $\H_{\G}$.
\end{lem}

Remember that the set $\S(\G)$ is defined by \eqref{short}, while $\M_k(\G)$ is defined by \eqref{mediumk} for $k>0$, and \eqref{medium0} for $k=0$.

\begin{lem}\label{AdmissibleGap2} 
Let $V \in \K(\H_{\G})$ be as in \eqref{GeneralPotential} and $k \in \Z_+$. If $V_j \in \S(\G) + \M_k(\G)$ for all $j=0,1,\ldots, N$, then $V \in \mathcal{C}^{1,1}\left(A_k^{(\G)}\right)$. 
\end{lem}

\begin{proof}
We first show that if $W \in \S(\G)$, then  $W \in \K(\H_{\G}) \cap \mathcal{C}^{1,1}\left(A_k^{(\G)}\right).$ By Lemma \ref{AppendixLemma} we know that $(n \left\|W(n)\right\|)_{n \in \G}$ is a bounded sequence. This  implies that $\lim_{n \to \infty}  \left\|W(n)\right\|=0$ and hence  $ W \in \K(\H_{\G})$. Furthermore, by Lemma \ref{AtoX} and  \cite[Theorem 7.5.8]{abmg}, we deduce that  $ W \in \mathcal{C}^{1,1}\left(A_k^{(\G)} \right)$ for all $k \in \Z_+$.

We now show that if $W \in \M_k(\G)$ then $W$ is of class $ C^1\left(A_k^{(\G)} \right)$ and $\left[ iA_k^{(\G)} , W \right] \in \Cc^{0,1}\left(A_k^{(\G)} \right)$. In particular, $W \in \Cc^{1,1}\left(A_k^{(\G)} \right)$.  By \eqref{FirstCommutatorAk} and \eqref{CommutatorA0}, on the space $ \ell^2_c(\G; \C^2)$ one has that
\begin{equation}\label{Eq2Appex}
     \left[ iA_k^{(\G)}  , W \right]=  SD_{k,1}(W) + D_{k,0}(W) + D_{k,-1}(W)S^*,
\end{equation}
where 
\begin{align*}
    D_{k,1}(W)&:= \left(c_{k,1} +d_{k,1}X^{(\G)}\right)F_{k,1}(W), \\
    D_{k,0}(W)&:= d_{k,0}X^{(\G)}F_{k,0}(W),\\
    D_{k,-1}(W)&:= \left(c_{k,-1} +d_{k,-1}X^{(\G)}\right)F_{k,-1}(W),
\end{align*}
are multiplication operators on $\H_{\G}$. Here $c_{k,l}$ and $ d_{k,l}$ are constants that do not depend on $W$,  $F_{k,1}(W)= F_{k,-1}(W) := \tau^{-k}W-W,$ $F_{k,0}(W):=0$ for $k>0$, and $ F_{0,l}(W) :=B_{l}(W) $ are given by \eqref{Dzero1} to \eqref{Dzero3} for $l=-1,0,1$. Since $W \in \M_k(\G)$, there are $0< \beta < \gamma < \infty$ such that
\begin{equation}\label{Eq3Appex}
        \int_1^\infty \sup_{ \beta r < |n| <  \gamma r}  \| F_{k,l}W(n) \| \, dr < \infty, \quad l=-1,0,1.
    \end{equation}
By Lemma \ref{AppendixLemma}, this implies that  $D_{k,l}(W) \in \B\left(\H_\G\right)$ for all $l=-1, 0, 1$. Thus, the identity \eqref{Eq2Appex} can be extended to the whole space $\H_\G$. This shows that $W \in C^1 \left(A_k^{(\G)} \right)$. 

Now let $\Theta \in C^\infty(\R)$ with $\theta(x)>0$ if $\beta < x < \gamma$ and $\theta(x)=0$ otherwise. Then by \eqref{Eq3Appex} we have that
\begin{align*}
    \int_1^\infty \left\| \Theta \left( {\langle X^{(\G)} \rangle}/{r} \right) \left[ iA_k^{(\G)}  , W \right]  \right\| \, \frac{dr}{r} & \leq \sum_{l=-1}^1   \int_1^\infty \left\| \Theta \left( {\langle X^{(\G)} \rangle}/{r} \right) D_{k,l}(W)  \right\| \, \frac{dr}{r} \\
    & \leq \gamma \int_1^\infty \sup_{ \beta r < |n| <  \gamma r}  \frac{1}{n}  \left \| D_{k,l}(W)  (n) \right\|  \, dr <  \infty.
\end{align*}
Hence by Lemma \ref{AtoX}, we can apply \cite[Theorem 7.5.8]{abmg} or \cite[Theorem 6.1]{B-Sah2} to deduce that $\left[iA_k^{(\G)} , W\right] \in \Cc^{0,1}\left(A \right)$. By Remark \ref{RmkC01}, we conclude that $W \in \Cc^{1,1}\left(A_k^{(\G)} \right)$. 

Therefore, since $\mathcal{C}^{1,1}\left(A_k^{(\G)}\right)$ is a Banach $*$-subalgebra of $\B\left(\H_{\G}\right)$, the claimed result follows from Lemma \ref{shiftRegularity} and the preceding argument.
\end{proof}


\section{Proofs of the main results}\label{proofss}

In this section, we prove our main results. For each $k \in \Z_+$, we  let $A_k^{(\G)}$ be as in Section~\ref{admissiblesboth}.

\begin{proof}[Proof of Theorem \ref{main1}.]
According to Theorems \ref{MainAk}, \ref{MourreH0+} and \ref{MainAkUnilateral} we know that $H_0^{(\G)}$ is of class $ \mathcal{C}^{1,1}\left(A_k^{(\G)}\right)$,  $$ \mu_k\left(H_0^{(\G)}\right) = {\mu}_{A_k^{(\G)}}\left(H_0^{(\G)}\right)  \quad \mbox{and}\quad \kappa_k\left(H_0^{(\G)}\right) = \kappa_{A_k^{(\G)}}\left(H_0^{(\G)}\right), $$ 
where the sets on the left-hand side of these equalities are given in \eqref{mukThm2.1} and  \eqref{CriPointsAk}, while those on the right-hand side are given in \eqref{muthetaj} and \eqref{kthetaj}.
Furthermore, from Lemmas \ref{ClasC1Gap} and \ref{AdmissibleGap2}, we deduce that under Assumption \ref{A1}, $V \in \B\left( H_{\G} \right)$ is an admissible perturbation of $H_0^{(\G)}$ with respect to $A_k^{(\G)}$. 
Therefore, the given statements follow from Proposition \ref{MainMourrePerturbed} and Lemma~\ref{AtoX}. This finishes the proof.
\end{proof}

\begin{proof}[Proof of Theorem \ref{main2}.] According to Theorems   \ref{MainA0} and \ref{MainA0Unilateral}, $H^{(\G)}_0 \in \mathcal{C}^{1,1}\left(A_0^{(\G)}\right),$
$$(-2|a|,2|a|) =\tilde{\mu}_{A_0^{(\G)}}\left(H_0^{(\G)}\right) ={\mu}_{A_0^{(\G)}}\left(H_0^{(\G)}\right)  \quad \mbox{and} \quad  \{ \pm 2|a| \} = \kappa_{A_0^{(\G)}}\left(H_0^{(\G)}\right).$$

By Lemma \ref{AdmissibleGap2}, if $V$ satisfies Assumption \ref{A2}, then $V \in \B\left( H_ {\G} \right)$ is an admissible perturbation of $H_0^{(\G)}$ with respect to $A_0^{(\G)}$. Therefore, the stated results follow from Proposition \ref{MainMourrePerturbed} and Lemma \ref{AtoX}. This concludes the proof. 
\end{proof}


\section{Appendix}\label{Appendix}

The goal of this appendix is to show a technical lemma that allows us to relate the sets $\S(\G)$ and $\M_k(\G)$, defined in Section \ref{Perturbations}, with the classes of admissible perturbations for $H_0^{(\G)}$ with respect to the conjugate operators defined in  Section \ref{MourreEstimatesFreeModel}. We adopt the convention that the supremum over the empty set is zero.
\begin{lem}\label{AppendixLemma}
Let $(a_n)_{n \in \Z_+}$ be a sequence of non-negative real numbers. Assume that there exist   $0<\beta <\gamma <  \infty$ such that
\begin{equation*}\label{integral}
    L_{\beta,\gamma}:= \int_{\frac{1}{\gamma-\beta}}^\infty \sup_{ \beta r < n <  \gamma r} a_n  \, dr = \int_{1}^\infty \sup_{ \beta r < n <  \gamma r} a_n  \, dr < \infty.
\end{equation*}
Then 
\begin{itemize}
    \item[(i)] $L_{\mu, \nu}< \infty$ for all $0<\mu < \nu$.
    \item[(ii)] The sequence $(na_n)_{n \in \Z_+}$ is bounded.
    \item[(iii)] The series $\sum_{n=1}^\infty a_n$ converges.
\end{itemize}
\end{lem}

\begin{proof} 
We start by showing (i). Let $c:= \frac{\gamma}{\beta}$ and $0<\mu < \nu$. Suppose that $ \frac{\nu}{\mu}= c$. By the change of variables $t=\frac{\beta r}{\mu}, $ we get 
\begin{equation}\label{Eqcase1}
    L_{ \mu  , \nu} =  \int_{1}^\infty \sup_{ \mu t < n <  \nu t} a_n dt = \frac{\beta}{\mu} \int_{\frac{1}{\gamma-\beta}}^\infty \sup_{ \beta r < n <  \gamma r} a_n dr
= \frac{\beta}{\mu}  L_{\gamma, \beta}. 
\end{equation}
Now we consider the general case. Let $m \in  \Z_+$ be such that $\nu< \mu c^m $. Then 
\begin{align*}
    \sup_{ \mu t < n <  \nu t} a_n \leq \sup_{ \mu t < n <  \mu c^m  t} a_n \leq \sum_{k=0}^{m-1} \sup_{ \mu c^{k} t < n < \mu c^{k+1} t} a_n, 
\end{align*}
 provided that $t> \frac{1}{\nu-\mu}$  and $t \notin \left\{ \frac{n}{\mu c^{k+1} } : k \in \{0, \ldots, m-1\},  n \in \Z_+    \right\}.$ Combining this with \eqref{Eqcase1}, we conclude that $L_{\mu, \nu}< \infty$. 

To show (ii) and (iii), fix $K \in \Z_+$ such that $K>{\frac{\gamma}{\gamma-\beta}}$. Note that if $k \in \Z_+$, $k \geq K$ and  $ r \in \left( \frac{k}{\gamma} , \frac{k}{\beta} \right)$, then $\beta r < k <  \gamma r$. Hence
$$
\left( \frac{1}{\beta} - \frac{1}{\gamma} \right) k a_k \leq  \int_{\frac{k}{\gamma}}^{ \frac{k}{\beta}} \sup_{ \beta r < n <  \gamma r} a_n  \, dr \leq L_{\beta, \gamma},
$$

which proves (ii). 

Finally,  for all $k \in   \Z_+ $, with $k \geq K$, let $r_k \in \left( \frac{k}{\gamma} , \frac{k}{\beta} \right)$  be such that
$$ a_k \leq \sup_{ \beta r_k < n <  \gamma r_k} a_n \leq \inf \left\{ \sup_{ \beta r < n <  \gamma r} a_n : r \in  \left( \frac{k}{\gamma} , \frac{k}{\beta} \right) \right\} + \frac{1}{2^k}. $$
Then
$$ \sum_{k=K}^\infty a_k \leq   \sum_{k=K}^\infty \int_{\frac{k}{\gamma}}^{ \frac{k}{\beta}} \sup_{ \beta r < n <  \gamma r} a_n  \, dr  + \sum_{k=K}^\infty \frac{1}{2^k}.$$
This completes the proof.
\end{proof}


\bibliographystyle{plain}  

\begin{thebibliography}{10}

\bibitem{ales}
R.I. Aleskerov.
\newblock An application of the inverse scattering problem for the discrete {D}irac operator.
\newblock {\em Proc. Inst. Math. Mech. Natl. Acad. Sci. Azerb.}, 46(1):94--101, 2020.

\bibitem{abmg}
W.O. Amrein, A.~Boutet de~Monvel, and V.~Georgescu.
\newblock {\em $C^0$-groups, commutator methods and spectral theory of N-body Hamiltonians}.
\newblock Birkh\"auser, Basel, 1996.

\bibitem{o-toe}
M.A. Astaburuaga, O.~Bourget, and V.H. Cortés.
\newblock Spectral stability for compact perturbations of {T}oeplitz matrices.
\newblock {\em J. Math. Anal. Appl.}, 440(2):885--910, 2016.

\bibitem{b-s}
E.~Bairamov and S.~Solmaz.
\newblock Spectrum and scattering function of the impulsive discrete {D}irac systems.
\newblock {\em Turkish J. Math.}, 42(6):3182--3194, 2018.

\bibitem{b-mf-t}
O.~Bourget, G.R.~Moreno Flores, and A.~Taarabt.
\newblock One-dimensional discrete {D}irac operators in a decaying random potential {I}: spectrum and dynamics.
\newblock {\em Math. Phys. Anal. Geom.}, 23(2):20, 2020.

\bibitem{B-Sah2}
A.~Boutet~de Monvel and J.~Sahbani.
\newblock On the spectral properties of discrete {S}chr\"odinger operators: the multi-dimensional case.
\newblock {\em Rev. Math. Phys.}, 11(9):1061--1078, 1999.

\bibitem{c-d-p}
S.~Carvalho, C.~de~Oliveira, and R.~Prado.
\newblock Sparse one-dimensional discrete {D}irac operators {II}: Spectral properties.
\newblock {\em J. Math. Phys.}, 52(7):073501, 2011.

\bibitem{c-d-p2}
S.~Carvalho, C.~de~Oliveira, and R.~Prado.
\newblock Dynamical localization for discrete {A}nderson {D}irac operators.
\newblock {\em J. Stat. Phys.}, 167:260--296, 2017.

\bibitem{c-i-k-s}
B.~Cassano, O.O. Ibrogimov, D.~Krej\v{c}i\v{r}ík, and F.~\v{S}tampach.
\newblock Location of eigenvalues of non-selfadjoint discrete {D}irac operators.
\newblock {\em Ann. Henri Poincaré}, 21(7):2193--2217, 2020.

\bibitem{cn-g-per}
A.~H. Castro~Neto, F.~Guinea, N.~M.~R. Peres, K.~S. Novoselov, and A.~K. Geim.
\newblock The electronic properties of graphene.
\newblock {\em Rev. Mod. Phys.}, 81:109--162, 2009.

\bibitem{P.C}
P.~A. Cojuhari.
\newblock On the spectrum of a class of block {J}acobi matrices.
\newblock In {\em Operator theory, structured matrices, and dilations}, volume~7 of {\em Theta Ser. Adv. Math.}, pages 137--152. Theta, Bucharest, 2007.

\bibitem{P.C2}
P.~A. Cojuhari.
\newblock Discrete spectrum in the gaps for perturbations of periodic {J}acobi matrices.
\newblock {\em J. Comput. Appl. Math.}, 225(2):374--386, 2009.

\bibitem{P.C-sah}
P.~A. Cojuhari and J.~Sahbani.
\newblock Spectral and scattering theory for perturbed block {T}oeplitz operators.
\newblock {\em arXiv preprint arXiv:1707.05183}, 2017.

\bibitem{c-g-j}
H.~Cornean, H.~Garde, and A.~Jensen.
\newblock Discrete approximations to {D}irac operators and norm resolvent convergence.
\newblock {\em J. Spectr. Theory}, 12(4):1589--1622, 2022.

\bibitem{g-g}
V.~Georgescu and S.~Gol\'enia.
\newblock Isometries, fock spaces and spectral analysis of {S}chr\"odinger operators on trees.
\newblock {\em J. Funct. Anal.}, 227:389--429, 2005.

\bibitem{goh}
I.~Gohberg, S.~Goldberg, and M.~A. Kaashoek.
\newblock {\em Classes of Linear Operators}, volume~49 of {\em Operator Theory, Advances and Applications}.
\newblock Birkh\"auser, 1990.

\bibitem{g-h}
S.~Gol\'enia and T.~Haugomat.
\newblock On the a.c. spectrum of the 1{D} discrete {D}irac operator.
\newblock {\em J. Methods Funct. Anal. Topology}, 20(3):252--273, 2014.

\bibitem{g-m-1}
S.~Gol\'enia and M.~Mandich.
\newblock Propagation estimates for one commutator regularity.
\newblock {\em J. Integr. Equ. Oper. Theory}, 62:90:47, 2018.

\bibitem{g-m-2}
S.~Gol\'enia and M.~Mandich.
\newblock Bands of pure absolutely continuous spectrum for lattice {S}chr\"odinger operators with a more general long range condition.
\newblock {\em J. Math. Phys.}, 62(9):Paper No. 092104, 32, 2021.

\bibitem{hul}
A.~Hulko.
\newblock On the number of eigenvalues of the discrete one-dimensional {D}irac operator with a complex potential.
\newblock {\em Anal. Math. Phys.}, 9(1):639--654, 2019.

\bibitem{kopylova}
E.~Kopylova.
\newblock Limiting absorption principle for the 1d discrete {D}irac equation.
\newblock {\em Russ. J. Math. Phys.}, 22(1):34--38, 2015.

\bibitem{k-t2}
E.~Kopylova and G.~Teschl.
\newblock Dispersion estimates for one-dimensional discrete {D}irac equations.
\newblock {\em J. Math. Anal. Appl.}, 434(1):191--208, 2016.

\bibitem{k-t}
E.~Kopylova and G.~Teschl.
\newblock Scattering properties and dispersion estimates for a one-dimensional discrete {D}irac equation.
\newblock {\em Math. Nachr.}, 295(4):762--784, 2022.

\bibitem{mandich2017}
M.~Mandich.
\newblock The limiting absorption principle for the discrete {W}igner-von {N}eumann operator.
\newblock {\em J. Funct. Anal.}, 272(6):2235--2272, 2017.

\bibitem{EM}
E.~Mourre.
\newblock Absence of singular continuous spectrum for certain self-adjoint operators.
\newblock {\em Comm. Math. Phys.}, 78(3):391--408, 1980/81.

\bibitem{Nakamura2024}
Shu Nakamura.
\newblock Remarks on discrete dirac operators and their continuum limits.
\newblock {\em J. Spectr. Theory}, 14(1):255--269, 2024.

\bibitem{n-7}
K.~S. Novoselov, A.~K. Geim, S.~V. Morozov, D.~Jiang, Y.~Zhang, S.~V. Dubonos, I.~V. Grigorieva, and A.A. Firsov.
\newblock Electric field effect in atomically thin carbon films.
\newblock {\em Science}, 306(5696):666--669, 2004.

\bibitem{rod}
L.~Rodman.
\newblock On the structure of self-adjoint {T}oeplitz operators with rational matrix symbols.
\newblock {\em Proc. Amer. Math. Soc.}, 92(4):487--494, 1984.

\bibitem{sahahah}
J.~Sahbani.
\newblock Spectral theory of a class of block {J}acobi matrices and applications.
\newblock {\em J. Math. Anal. Appl.}, 438(1):93--118, 2016.

\bibitem{SSH79}
W.~P. Su, J.~R. Schrieffer, and A.~J. Heeger.
\newblock Solitons in polyacetylene.
\newblock {\em Phys. Rev. Lett.}, 42:1698--1701, 1979.

\bibitem{T}
Y.~Tadano.
\newblock Long-range scattering theory for discrete {S}chr\"odinger operators on graphene.
\newblock {\em J. Math. Phys.}, 60(5):052107, 2019.

\end{thebibliography}

\phantomsection

\end{document}